\documentclass[a4paper,reqno]{amsart}
\usepackage{amssymb,bbm,amsmath,amsfonts,amsthm, amscd,bm}
\usepackage{xcolor}
\usepackage{lineno}
\usepackage{mathtools}
\usepackage{enumitem}
\usepackage{esint}
\usepackage{microtype}
\usepackage[linktocpage,colorlinks=true,raiselinks]{hyperref}
\hypersetup{urlcolor=blue, citecolor=red, linkcolor=blue}

\newtheorem{definition}{Definition}[section]
\newtheorem{proposition}{Proposition}[section]
\newtheorem{lemma}{Lemma}[section]
\newtheorem{theorem}{Theorem}[section]
\newtheorem{corollary}{Corollary}[section]
\newtheorem{remark}{Remark}[section]
\newcommand{\rplus}{\R^{\smash{3}}_{\smash{\scriptscriptstyle +}}}
\newcommand{\rplusd}{\R^{\smash{d}}_{\smash{\scriptscriptstyle +}}}

\newcommand{\bu}{{u}}

\newcommand{\bvn}{\mathbf{v}}
\newcommand{\bun}{\mathbf{u}}

\newcommand{\bvpsi}{{\bm\psi}}
\newcommand{\bvphi}{{\bm\varphi}}

\newcommand{\bx}{x}
\newcommand{\by}{y}

\newcommand{\bphi}{\boldsymbol{\varphi}}
\newcommand{\vep}{\varepsilon }

\newcommand{\de}{\mathrm{d}}
\newcommand{\dx}{\,\de\bx}
\newcommand{\dy}{\,\de\by}

\newcommand{\R}{\mathbbm{R}}
\newcommand{\T}{\mathbbm{T}}

\newcommand{\tore}{{\mathbbm{T}^{3}}}
\DeclareMathOperator{\dive}{div}

\numberwithin{equation}{section}

\begin{document}


\title[Energy conservation: the H\"older case, with Dirichlet bc]{\textbf{Energy conservation for weak solutions of incompressible
    Newtonian fluid equations\\ in H\"older spaces with Dirichlet
    boundary conditions in the half-space} }
  \author{Luigi  C. Berselli}
  \address[L.C. Berselli]{Dipartimento di Matematica, Universit\`a
    di Pisa, Via F. Buonarroti 1/c, I-56127 Pisa,  ITALY, email:
    luigi.carlo.berselli@unipi.it}
  \author{Alex Kaltenbach}
  \address[A. Kaltenbach]{Institute of Mathematics, Technical
    University of Berlin, Stra\ss e des 17 Juni 136, D-10623 Berlin,
    GERMANY, email:kaltenbach@math.tu-berlin.de}
  \author{Michael R\r u\v zi\v cka{}}
  \address[M. R\r u\v zi\v cka{}]{Department of Applied Mathematics,
    University of Freiburg, Ernst--Zermelo--Str.~1, D-79104
    Freiburg, GERMANY, email: {rose@mathematik.uni-freiburg.de},
  corresponding author}
  \dedicatory{Dedicated to the memory of Vsevolod A. Solonnikov}
  \begin{abstract}
    We investigate sufficient H\"older continuity conditions on Leray--Hopf (weak)
    solutions to the in unsteady Navier--Stokes equations in three dimensions guaranteeing energy conservation.
    Our focus is on the half-space case with homogeneous Dirichlet boundary
    conditions. This problem is more technically challenging, if compared to the Cauchy or
    periodic cases, and has not been previously addressed. At present are known a few
    sub-optimal results obtained through Morrey embedding results based on conditions for the
    gradient of the velocity in Sobolev spaces. Moreover, the results in this paper are
    obtained without any additional assumption neither on the pressure nor the flux of the
    velocity, near to the boundary.
   \end{abstract}
  
 \maketitle


\section{Introduction}
\thispagestyle{empty}

In the present paper, we establish results of energy conservation for Leray--Hopf (weak)
solutions of the unsteady Navier--Stokes equations in the half-space  with
homogeneous Dirichlet boundary conditions, in three space dimensions, under certain
additional assumptions of H\"older continuity. The velocity vector field
$\bvn \colon [0,T)\times \rplus\to \mathbbm{R}^3$ and kinematic pressure field
$\pi\colon (0,T)\times \rplus \to \mathbbm{R}$ are~such~that
\begin{equation}
  \label{eq:NSE}
  \begin{aligned}
    \partial_{t}\bvn-\nu\Delta\bvn+[\nabla\bvn]\bvn+\nabla\pi
    &=\mathbf{0}&&\quad\text{ in }(0,T)\times \rplus \,,
     \\
    \dive\bvn&=0&&\quad \text{ in }(0,T)\times \rplus \,,
    \\
    \bvn&=\mathbf{0}&&\quad \text{ on }(0,T)\times\partial\rplus\,,
    \\
    \bvn(0)&=\bvn_0 &&\quad \text{ in }\rplus \,,
  \end{aligned}
\end{equation}
where $\bvn_{0}\colon \rplus\to \mathbbm{R}^3$ is an initial condition
and  $\nu>0$ denotes the kinematic viscosity.  The convective term $\smash{[\nabla\bvn]\bvn\colon\rplus\to
  \mathbb{R}^3}$~is defined via
$([\nabla\bvn]\bvn)_i\hspace{-0.1em}\coloneqq \hspace{-0.1em}\smash{\sum_{j=1}^3{v_j\partial_j
    v_i}}$ for all $i\hspace{-0.1em}=\hspace{-0.1em}1,2,3$.

As suggested by Onsager's conjecture (\textit{cf}.~\cite{Ons1949}), we consider
Leray--Hopf (weak) solutions~to~the~unsteady Navier--Stokes equations~\eqref{eq:NSE}, where on the
velocity vector field is assumed~an additional H\"older regularity with respect to the space
variable.
Here, we focus on the boundary value problem, while the periodic case has been extensively
studied by various authors (\textit{cf}.~\cite{CL2020,Ber2023b,Ber2023}) under different
assumptions. The primary technical challenge involves constructing a divergence-free
approximation of the velocity vector field that respects the homogeneous Dirichlet boundary
condition and that allows for proper estimates and exact cancellations, see
Section~\ref{sec:mollification}. We achieve this through a combination of spatial
convolution and  translation. However, the resulting smoothing operator is not
self-adjoint, which introduces new technical challenges when studying terms involving
the time-derivative. To overcome these challenges, it becomes necessary to prove some
additional regularity on the time-derivative. To this end, we utilize results on
maximal parabolic regularity for the unsteady  Stokes equations (\textit{i.e.},  system \eqref{eq:NSE} without convective~term~$[\nabla \bvn]\bvn$). More precisely, this approach leads  to
certain estimates on the time-derivative (and second order derivatives) of the velocity at
the price of mild technical assumptions on the initial datum, especially regarding suitable
decay of fractional-order derivatives (being the domain with infinite measure).

It is worth noting that the same problem for the Euler equations (\textit{i.e.},~when~${\nu=0}$) allows for a simpler treatment (at least from~a~formal technical point of view)
as the pressure estimates (and, therefore, also those concerning the time~derivative) can
be obtained at a fixed time by solving certain steady~Poisson~\mbox{problems}~(\textit{cf}.~\cite{BBT2023,BT2018,BT2022}). Moreover, in the case of the Euler equations
in the half-space, one can reflect~the~data, as demonstrated in~\cite{RRS2018}, in this way, 
avoiding the need to deal with~the~pressure. However, the reflection technique is not
applicable to the unsteady Navier--Stokes equations with homogeneous Dirichlet boundary condition. It is
well-known that this can be done in the case of Navier conditions: a suitable extension
(\textit{i.e.}, \textit{[even, even, odd, even]} for the three components of the velocity and for the pressure)
allows to transform the problem into another equivalent problem in the full-space~$\mathbbm{R}^3$, thereby
obtaining the same results as those valid for the Cauchy problem. 

Coming the to specific assumptions of H\"older type on the velocity vector field, we recall that Onsager's conjecture (\textit{cf}.~\cite{Ons1949}), which has recently been confirmed for the Euler equations
(\textit{cf}.~\cite{Ise2018,BDLSV2019}), suggested the value $\alpha=\frac{1}{3}$ (\textit{cf}.~\cite{Ons1949}). \mbox{However}, the conjecture mainly focused on spatial H\"older
regularity. In our study, we consider the viscous problem (for which the convex
integration techniques are still not completely satisfactory) and establish families of
criteria based on the H\"older exponent in space (in combination with proper
time-integrability), as we have already explored for the space-periodic case
in~\cite{Ber2023b}.

We also recall that the classical scaling invariant results imply that if the velocity vector field satisfies 
$\bvn\in L^{2}(0,T;(L^{\infty}(\rplus))^3)$, then weak solutions~are~also~strong; this ensures
uniqueness, regularity, and energy conservation. As a consequence, the
results in this paper should involve time-integrability exponents strictly~less~than~two.

To provide context with recent literature (also see the review in~\cite{Ber2021}), it is
worth mentioning the recent results for the unsteady Navier--Stokes equations \eqref{eq:NSE} obtained by Cheskidov \textit{et
  al.}~\cite{CFS2010} in the setting of (fractional) Sobolev spaces and by Cheskidov and
Luo~\cite{CL2020}, along with Farwig and Taniuchi~\cite{FT2010}. The sharpest known result
in the Besov space setting appears to be as follows: suppose that $1\leq \beta<p\leq\infty$ are
such~that~$\frac{2}{p}+\frac{1}{\beta}<1$.~If
\begin{equation}
    \label{eq:Cheskidov-Luo}
    \bvn\in L^{\beta}_{w}(0,T;(B^{\frac{2}{p}+\frac{2}{\beta}-1}_{p,\infty}(\tore))^3)\,,    
\end{equation}
then energy conservation holds, where
$L^{\beta}_{w}\hspace{-0.175em}=\hspace{-0.175em}L^{\beta,\infty}\hspace{-0.1em}$
denotes the  Marcinkiewicz~\mbox{$L^{\beta}$-space} and
$B^{\alpha}_{pq}(\tore)$ denotes periodic Besov functions. This result includes as sub-cases the
classical results by Prodi~\cite{Pro1959}~and~Lions~\cite{Lio1960}, and
$\bvn\in L^{\beta}_{w}(0,T;\smash{(B^{\smash{1/3}}_{3,\infty}(\tore))^3})$, for all $\beta>3$, represents the
borderline case. Recent extensions are based on conditions on the gradient 
$\nabla\bvn\in L^{\smash{5q/(5q-6)}}(0,T;(L^{q}(\Omega))^{3\times 3})$, where $q>\frac{9}{5}$ and  $\Omega\subseteq\R^{3}$,
see~\cite{BC2020} and see Beir\~ao da Veiga and Yang~\cite{BY2019}, for results in a
wider range of exponents. The latter results, which hold true for the case with homogeneous
Dirichlet condition, can be translated (resorting to the Sobolev--Morrey embedding) in terms of
H\"older spaces as follows:
\begin{equation*}
    \bvn \in L^{\frac{5}{3+2\alpha}}(0,T;(C^{0,\alpha}(\Omega ))^3)\,,    
\end{equation*}
where $\alpha\in (0,1]$ and $\Omega\subseteq \mathbbm{R}^3$,
which have the same scaling as condition~\eqref{eq:Cheskidov-Luo}.

Recently, in~\cite[\S~5]{Ber2023b}, we established in the space-periodic case results
directly in the class of H\"older continuous functions proving as sufficient condition for
the energy conservation that the velocity vector field satisfies
\begin{equation}
  \label{eq:holder-periodic}
  \bvn\in L^{\frac{2}{1+\alpha}}_{\textup{loc}}(0,T;({\dot C}^{0,\alpha}(\tore))^3)\,.
\end{equation}
We also considered the half-space case with Dirichlet conditions proving sub-optimal
results with partial H\"older regularity (\textit{i.e.}, H\"older regularity only with respect to
two variables), \textit{cf}.~\cite[\S~6]{Ber2023b}.

Now we can formulate the main result for weak solutions to the Dirichlet boundary initial value
problem~\eqref{eq:NSE}. To do so we introduce the
space $X_{s,r}$, which is defined as the closure of smooth and
compactly supported divergence-free functions in $\rplus$, with
respect to the Besov norm
of $B^{2-2/r}_{rs}(\rplus)$.
\begin{theorem}
  \label{thm:main}
  Let $T \in (0,\infty)$,  $\bvn_0\in L^2_{0,\sigma}(\rplus)$, and let  
  $\bvn\in L^2(0,T;W^{1,2}_{0,\sigma}(\rplus))$
  $\cap \, L^\infty(0,T;L^2_{0,\sigma}(\rplus))$ be a corresponding Leray--Hopf (weak) solution (in the sense of Definition \ref{def:leray_hopf}). 
  If for some $\alpha\in (0,1)$  and some $\beta>\beta_0=\frac{6}{3+2\alpha}$ 
  \begin{align}
    \bvn\in    L^{\beta}(0,T;(C^{0,\alpha}(\rplus))^3)\,,
    \label{eq:Holder1-3}
  \end{align}
  and for some $r\in(\frac{4}{2+\beta},\frac{4}{4-\beta})\ 
 \text{and } s:=\frac{2r\beta}{4r+r \beta-4}$
  \begin{equation}
  \bvn_{0}\in X_{s,r}\,,
  \label{eq:Holder1-3.2} 
\end{equation}
%
  then it follows that for every
  $t\in [0,T]$, holds the global~energy~equality
  \begin{align}
    \label{eq:energy-equality}
    \frac{1}{2}\|\bvn(t)\|_{L^2(\rplus)}^{2}+\nu \int_{0}^{t}{\|\nabla
    \bvn(\tau)\|^{2}_{L^2(\rplus)}\,\mathrm{d}\tau}=\frac{1}{2}\|\bvn(0)\|_{L^2(\rplus)}^{2}\,. 
  \end{align}
\end{theorem}
%
The hypotheses on the initial datum could alternatively be also 
expressed by means of a power of $A_{r}$, the Stokes operator acting in
a subset of $(L^{r}(\rplus))^3$, as defined by Sohr
(\textit{cf}.~\cite{Soh2001}). The extra assumption on the initial
datum is needed to ensure proper summability of the solution in an
unbounded domain and does not
imply that the solution belongs to a regularity class.

It is important to mention
that %
the treatment of the convective term allows, in the case of a (homogeneous) Dirichlet boundary condition, for the same exponents
as in the space-periodic case~\eqref{eq:holder-periodic}. However, the technical
obstruction to achieve, in the case of a (homogeneous) Dirichlet boundary condition, the same results as
in~\eqref{eq:holder-periodic} arises from the estimates we are forced to obtain on the
time-derivative (and, consequently, on the gradient of the pressure), as done in
Proposition~\ref{prop:timed}. Moreover, the condition on the initial datum \eqref{eq:Holder1-3.2} is due to the
unbounded domain. In the case of a bounded domain, the assumptions on the initial datum can
be relaxed and one can also handle in the time variable the ``local'' space~$L^{\beta}_{\textup{loc}}$~in~\eqref{eq:Holder1-3}. However, this result needs a different and more technical
approach~to~the~smoothing, see the forthcoming paper~\cite{BKR2024b}, to overcome all the
technical problems and restrictions which arise when trying to regularize divergence-free
functions vanishing at the boundary, as pointed out already in~\cite{Ber2023b}.

\medskip

\noindent\textbf{Plan of the paper:} In Section~\ref{sec:preliminaries}, we
set the notation and recall some basic~\mbox{results}. In Section~\ref{sec:regularity}, we
present an alternative weak formulation of the unsteady Navier--Stokes equations \eqref{eq:NSE} and  discuss the improved regularity of the time
derivative which can be obtained using parabolic maximal regularity results for the
unsteady Stokes equations. In Section~\ref{sec:mollification}, we introduce the smoothing by
convolution-translation which preserves simultaneously the incompressibility constraint and the homogeneous
Dirichlet boundary condition. Basic properties of this smoothing operator~are~established. Eventually, in Section~\ref{sec:NSE}, the main result of the paper about energy
conservation for  Leray--Hopf (weak) solutions that are  H\"older continuous with respect to the space variable  is proved.
%
%
\section{Notation and Preliminaries}
\label{sec:preliminaries}
In the sequel, by $\rplusd\coloneqq \{x_d> 0\}$, $d\in \mathbbm{N}$, we will denote the half-space. Vectors and
tensors are denoted by boldface letters.  For $r\in [1,\infty]$ and an open set $\Omega\subseteq \mathbbm{R}^d$, $d\in \mathbbm{N}$, 
we will use classical Lebesgue $L^r(\Omega)$ and Sobolev $W^{k,r}(\Omega)$ spaces, respectively. 
Moreover, for an open set $\Omega\subseteq \mathbbm{R}^d$, $d\in \mathbbm{N}$, we will~employ~the~vector~space~of $\alpha$-H\"older continuous functions $\dot{C}^{0,\alpha}(\overline{\Omega})$,
where $\alpha \in (0,1]$, which is a subset of~the~vector~space of continuous functions $u\in C^0(\overline{\Omega})$ such
that~the~\mbox{semi-norm}
\begin{align*}
  [u]_{\dot{C}^{0,\alpha}(\overline{\Omega})}\coloneqq \sup_{x,y\in \overline{\Omega}
  \;:\;
  x\neq y}{\frac{\vert u(x)-u(y)\vert }{\vert x-y\vert^{\alpha}}}\,,
\end{align*}
is finite. We mainly consider Banach space of bounded $\alpha$-H\"older continuous
functions
$C^{0,\alpha}(\overline{\Omega})\coloneqq C_{b}^0(\overline{\Omega})\cap \dot{C}^{0,\alpha}(\overline{\Omega})$,
where $\alpha \in (0,1)$, equipped with the norm
\begin{align*}
  \|u\|_{C^{0,\alpha}(\overline{\Omega})}\coloneqq \sup_{x\in \overline{\Omega}}{\vert u(x)\vert }
  +[u]_{\dot{C}^{0,\alpha}(\overline{\Omega})}\,.
\end{align*}
To define the notions of a weak solution, for $r\in [1,\infty)$ and an~open~set~$\Omega\subseteq \mathbbm{R}^d$, $d\in \mathbbm{N}$, 
we denote by $L^r_{0,\sigma}(\Omega)$ and $W^{1,r}_{0,\sigma }(\Omega)$, the closures of smooth,
compactly supported, and divergence-free vector fields in $(L^{r}(\Omega))^{d}$ and
$(W^{1,r}_{0}(\Omega))^{d}$, respectively. The spaces
$B^{\alpha}_{pq}(\R^{d})$, with $\alpha>0$ and $1<p,q<\infty$,  are the customary Besov spaces, and for the
half-space we set, as in \cite{kang-aniso},
\begin{equation*}
B^{\alpha}_{pq}(\R^d_+):=\left\{rf \; \big|\; f\in B^{\alpha}_{pq}(\R^{d})\right\},
\end{equation*}
where $rf$ denotes the restriction operator from $\R^{d}$ to
$\R^d_+$. The norm is defined as
$\|f\|_{B^{\alpha}_{pq}(\R^d_+)}:=\inf\|F\|_{B^{\alpha}_{pq}(\R^{d})}$,
where the 
infimum is taken over $F\in {B^{\alpha}_{pq}(\R^{d})}$ such that $rF=f$. Finally $X_{s,r}=\overline{((C^\infty_{0,\sigma}(\rplus))^3}^{\|\,.\,\|_{B^{2-2/r}_{rs}}}$.


        Throughout the entire paper, we assume that $T\in (0,\infty)$.
        
Let $X_1$, $X_2$ be Banach spaces such that there exists an embedding $j\colon X_1\to X_2$. Then, a function $u\in L^1_{\textup{loc}}(0,T;X_1)$ is said to have a generalized time derivative with respect~to $j$ if
there exists a function $v\in L^1_{\textup{loc}}(0,T;X_2)$ such that for every $\varphi\in C^\infty_0(I)$, it holds that
\begin{align}
    j\bigg(-\int_0^T{u(t)\,\varphi'(t)\,\mathrm{d}t}\bigg)=\int_0^T{v(t)\,\varphi(t)\,\mathrm{d}t}\quad \text{ in }X_2\,.
\end{align}
In this case, we define $\frac{\mathrm{d}_ju}{\mathrm{d}t}\coloneqq v$. 
For $p,q\in [1,+\infty]$, the Bochner--Sobolev space is defined via
\begin{align*}
    W^{1,p,q}_j(0,T;X_1,X_2)\coloneqq \bigg\{u\in L^p(0,T;X_1) \; \Big|\;\exists \;\frac{\mathrm{d}_ju}{\mathrm{d}t}\in L^q(0,T;X_2)\bigg\}\,.
\end{align*}
In the case $X_1= X_2$, $j=\textup{id}_{X_1}$, and $p=q\in [1,+\infty]$, we abbreviate $\frac{\mathrm{d}}{\mathrm{d}t}\coloneqq\frac{\mathrm{d}_{\textup{id}_{X_1}}}{\mathrm{d}t}$ and $W^{1,q}(0,T;X_1)\coloneqq W^{1,p,q}_{\textup{id}_{X_1}}(0,T;X_1,X_1)$.

\begin{proposition}[non-symmetric integration-by-parts formula]\label{prop:non_sym_pi}
	Let $(X,Y,i)$ be an evolution triple, \textit{i.e.}, $X$ is a Banach space, $Y$ is a Hilbert space with Riesz isomorphism $R_Y\colon Y\to Y^*$, and $i\colon X\to Y$ is  a dense embedding. Moreover, let $j\coloneqq i^*\circ R_Y\circ i\colon Y\to Y^*$ be the canonical embedding and let $1<p\leq q<\infty$. Then, the following statements apply:
	\begin{itemize}
		\item[(i)] \textup{(Weakly) continuous representation:} for each  $u\in W^{1,p,q'}_j(0,T;X,X^*)$ such that the function  $iu\in L^p(0,T;Y)$, defined by $(iu)(t)\coloneqq i(u(t))$ in $Y$ for a.e.\ $t\in (0,T)$, satisfies $iu\in L^\infty(0,T;Y)$ has a weakly continuous representation $i_\omega u\in C^0_\omega(0,T;Y)$. Moreover, for each $u\in W^{1,q,p'}_j(0,T;X,X^*)$,  the function  $iu\in L^p(0,T;Y)$ has a continuous representation $i_c u\in C^0(0,T;Y)$.
		
		\item[(ii)] \textup{Integration-by-parts formula:} For every  $u\in W^{1,p,q'}_j(0,T;X,X^*)$ such that $iu\in L^\infty(0,T;Y)$, $v\in W^{1,q,p'}_j(0,T;X,X^*)$, and $t,t'\in [0,T]$, it holds that
		\begin{align*}
			\int_{t'}^t{\bigg\langle \frac{\mathrm{d}_j u}{\mathrm{d}t}(s),v(s)\bigg\rangle_X\,\mathrm{d}s}=[((i_\omega u)(s),(i_c v)(s))_Y]_{s=t'}^{s=t}
			-\int_{t'}^t{\bigg\langle \frac{\mathrm{d}_j v}{\mathrm{d}t}(s),u(s)\bigg\rangle_X\,\mathrm{d}s}\,.
		\end{align*}
	\end{itemize}
\end{proposition}

\begin{proof}
	See \cite[Proposition 2.28]{alex-book}.
      \end{proof}

Since $(W^{1,2}_{0,\sigma }(\Omega),L^2_{0,\sigma }(\Omega),\textup{id}_{W^{1,2}_{0,\sigma }(\Omega)})$ is an evolution triple, the canonical embedding 
\begin{align*}
	e\coloneqq (\textup{id}_{\smash{W^{1,2}_{0,\sigma }(\Omega)}})^*\circ R_{\smash{L^2_{0,\sigma }(\Omega)}}\circ\textup{id}_{\smash{W^{1,2}_{0,\sigma }(\Omega)}}\colon W^{1,2}_{0,\sigma }(\Omega)\to (W^{1,2}_{0,\sigma }(\Omega))^*\,,
\end{align*}
where $(\textup{id}_{\smash{W^{1,2}_{0,\sigma }(\Omega)}})^*\colon  (L^2_{0,\sigma }(\Omega))^*\to (W^{1,2}_{0,\sigma }(\Omega))^*$ is the adjoint operator of the identity mapping $\textup{id}_{\smash{W^{1,2}_{0,\sigma }(\Omega)}}\colon  W^{1,2}_{0,\sigma }(\Omega)\to L^2_{0,\sigma }(\Omega)$ and $R_{\smash{L^2_{0,\sigma }(\Omega)}}\colon L^2_{0,\sigma }(\Omega)\to (L^2_{0,\sigma }(\Omega))^*$ the Riesz isomorphism, is an embedding.

We also recall the definition of weak solution for the three-dimensional unsteady Navier-Stokes
equations \eqref{eq:NSE}.\enlargethispage{4mm}
\begin{definition}[Leray--Hopf (weak) solution]
  \label{def:leray_hopf}
  Let  $\bvn_0\hspace{-0.175em}\in \hspace{-0.175em} L_{0,\sigma}^2(\rplus)$. Then,~a~Lebesgue- measurable
  vector field $\bvn\colon (0,T)\times \rplus\to \mathbbm{R}^3$ is said to be a
  \textup{Leray--Hopf (weak) solution} to the unsteady Navier--Stokes
  equations~\eqref{eq:NSE} with initial datum $\bvn_0$ if \linebreak
  $\bvn\in W^{1,2,\smash{\frac{4}{3}}}_e(0,T;W^{1,2}_{0,\sigma}(\rplus),(W^{1,2}_{0,\sigma}(\rplus))^*)\cap L^\infty(0,T;L_{0,\sigma}^2(\rplus))$ with weakly continuous (not relabeled) representation $\bvn\in C^0_{\omega}([0,T];L^2_{0,\sigma}(\Omega))$
satisfies 
  \begin{itemize}
  \item[(i)] \textup{Weak formulation:} For every $\bphi\in L^4(0,T;W^{1,2}_{0,\sigma}(\rplus))$, 
  it holds that
    \begin{align}
      \label{def:leray_hopf.1}
      \int_0^T\bigg[\bigg\langle \frac{\de_e \bvn}{\de
    t},\bphi\bigg\rangle_{\hspace{-1mm}W^{1,2}_{0,\sigma}(\rplus)}+\nu\int_{\rplus}\nabla\bvn:\nabla\bphi\dx+\int_{\rplus}
  [\nabla\bvn]\bvn\cdot\bphi\dx\bigg]\de\tau=0\,.
    \end{align}
  \item[(ii)] \textup{Energy inequality:} For every $t\in [0,T]$, it holds that
    \begin{align}
      \label{def:leray_hopf.2}
      \frac{1}{2}\|\bvn(t)\|_{L^2(\rplus)}^2+\nu \int_0^t{\|\nabla
      \bvn(\tau)\|_{L^2(\rplus)}^{2}\,\mathrm{d}\tau}\leq
      \frac{1}{2}\|\bvn_0\|_{L^2(\rplus)}^2\,. 
    \end{align}
  \item[(iii)] \textup{Initial condition:} It holds that 
  \begin{align} \label{def:leray_hopf.3}
  	\bvn(0)=\bvn_0\quad\text{ in }L^2_{0,\sigma}(\Omega)\,.
  \end{align}
  \end{itemize} 
\end{definition}
\begin{remark}[Equivalent weak formulation]\label{rem:equiv_form}
	By the non-symmetric integration-by-parts formula (\textit{cf}.\ Proposition \ref{prop:non_sym_pi}),
	the weak formulation \eqref{def:leray_hopf.1} together with the initial condition \eqref{def:leray_hopf.3} is equivalent to the following weak formulation: 
	for every $\bphi\in W^{1,4,2}_e(0,T;W^{1,2}_{0,\sigma}(\rplus),(W^{1,2}_{0,\sigma}(\rplus))^*)$ with $\bphi(T)=\mathbf{0}$ in $L^2_{0,\sigma}(\Omega)$, it holds that
	  \begin{align*} 
		&\int_0^T\bigg[\bigg\langle -\frac{\de_e \bphi}{\de
			t},\bvn\bigg\rangle_{\hspace{-1mm}W^{1,2}_{0,\sigma}(\rplus)}+\nu\int_{\rplus}\nabla\bvn:\nabla\bphi\dx+\int_{\rplus}
			[\nabla\bvn]\bvn\cdot\bphi\dx\bigg]\de\tau\\&=\int_{\rplus}\bvn_0\cdot \bphi(0)\dx\,.
	\end{align*}
\end{remark}
It is well-known (\textit{cf}.~\cite[Theorem 3.1]{Gal2000a}) that for any
$T \in(0,\infty)$ and
${\bvn_0\hspace*{-0.1em}\in\hspace*{-0.1em} L^2_{0,\sigma}(\rplus)}$, there exists a
Leray--Hopf (weak) solution in the sense of Definition~\ref{def:leray_hopf}.  The energy
inequality~\eqref{def:leray_hopf.2} can equally be re-written as the following equality:
for~every~$t\in [0,T]$, it holds that
\begin{equation*}
  \frac{1}{2}\|\bvn(t)\|_{L^2(\rplus)}^{2}+\nu
  \int_0^t{\int_{\rplus}{\epsilon[\bvn](\tau,x)\,\mathrm{d}x}\,\mathrm{d}\tau}=
  \frac{1}{2}\|\bvn_0\|_{L^2(\rplus)}^{2}\,, 
\end{equation*}
where the \textit{total energy dissipation rate} is given by 
\begin{equation*}
	\epsilon[\bvn]\coloneqq \nu \vert \nabla\bvn\vert^2+D(\bvn)\geq0\,,
\end{equation*}
with $D(\bvn)\in M_+(\rplus)$ a non-negative Radon measure. If
${D(\bvn)=0}$, then the energy dissipation arises only from the term with
the viscosity and equality~\eqref{eq:energy-equality}~holds~true.
%
%
\section{Further regularity
  properties of weak solutions}
\label{sec:regularity}

In this section, we establish that given the regularity assumptions~\eqref{eq:Holder1-3} and~\eqref{eq:Holder1-3.2} on the velocity vector field
and  its initial datum, 
respectively, this Leray--Hopf (weak)
solution satisfies higher regularity properties  in mixed spaces
and~an~almost everywhere~\mbox{point-wise} formulation of the unsteady Navier--Stokes equations~\eqref{eq:NSE} (in the sense of Definition~\ref{def:leray_hopf}); see Theorem \ref{thm:giga}.
The main tool to establish the latter will be the maximal regularity for solutions of the
unsteady Stokes equations as, \textit{e.g.}, in Sohr and von
Wahl~\cite{SvW1986} or Chang and Kang \cite{kang-aniso}. For the reader's~convenience,
we give a detailed and alternative proof of this result, without using any
smoothing of the equations, which avoids fine technical problems in the definition of
the Yosida approximation (as in~\cite{SvW1986}) due to the fact that the domain~is~unbounded.

A  technical result we will employ is the following existence, uniqueness, and
norm estimates in mixed spaces for strong solutions to the unsteady Stokes problem.
\enlargethispage{7mm}
  \begin{theorem}
    \label{thm:linear-Stokes}
    Let $\mathbf{f}\in L^{s}(0,T;L^{r}_{0,\sigma}(\rplus))$ with $r,s \in (1,\infty)$, let
    $\bvpsi_0\in X_{s,r}$, where the norm is denoted by $|||
    \,.\,|||_{s,r}:=\|\,.\,\|_{B^{2-2/r}_{sr}}$ and define 
    \begin{equation*}
      W^{2,1}_{s,r}\coloneqq \big\{ \bvphi\in  L^{s}(0,T;(W^{2,r}(\rplus))^3)\mid \partial_t \bvphi\in L^s(0,T;(L^{r}(\rplus))^3)\big\}\,. 
    \end{equation*}
    Then, there exists a unique velocity vector field $\bvpsi\in W^{2,1}_{s,r}$ and a
    non-unique pressure $\lambda\in L^s(0,T;W^{1,r}(\rplus))$, such that
    \begin{equation}
      \label{eq:Stokes}
      \begin{aligned}
        \partial_{t}\bvpsi-\nu\Delta\bvpsi+\nabla \lambda
        &=\mathbf{f}&&\quad\text{ a.e.\ in }(0,T)\times \rplus \,,
        \\
        \dive\bvpsi&=0&&\quad \text{ a.e.\ in }(0,T)\times \rplus \,,
        \\
        \bvpsi&=\mathbf{0}&&\quad \text{ a.e.\ on
        }(0,T)\times\partial\rplus\,,
        \\
        \bvpsi(0)&=\bvpsi_0 &&\quad \text{ a.e.\ in }\rplus \,,
      \end{aligned}
    \end{equation}
    and 
    \begin{equation*}
    \begin{aligned}
          & \int_0^T{\big\{\|\partial_t\bvpsi
      \|_{L^{r}(\rplus)}^{s}+\|\bvpsi\|_{W^{2,r}(\rplus)}^{s}+\|
      \lambda\|_{W^{1,r}(\rplus)}^{s}\big\}\,\de\tau}
    \\&\quad\leq c(r,s,T,\nu)\,\int_0^T\|\mathbf{f}\|_{L^{r}(\rplus)}^{s}\,\de\tau+|||\bvpsi_0|||_{s,r}^{s}\,.
    \end{aligned} 
    \end{equation*}
  \end{theorem}
  %
  The analysis of the unsteady Stokes problem \eqref{eq:Stokes} dates back to
  Solonnikov~\cite{Sol1964} for the case $r=s$. The mixed case $r\not= s$ have been
  studied --among the others and with different methods-- by Solonnikov~\cite{Sol1976},
  Yudovich~\cite{Yud1989}, Giga and Sohr~\cite{GS1991}, Maremonti and
  Solonnikov~\cite{MS1995}, and Chang and Kang \cite{kang-aniso}. Our
  formulation~is~taken from~\cite{kang-aniso}.\enlargethispage{5mm}

The main result we obtain from Theorem~\ref{thm:linear-Stokes} is the following
one (\textit{cf}.~\cite{SvW1986,kang-aniso}, where slightly different
arguments are used and only the case $\frac{2}{r}+\frac{3}{s}=4$ is considered).
\begin{theorem}
  \label{thm:giga}
  Let   
  $\bvn\in L^2(0,T;W^{1,2}_{0,\sigma}(\rplus))\cap L^\infty(0,T;L_{0,\sigma}^2(\rplus))$
 be a  Leray--Hopf  (weak) solution corresponding to  $\bvn_0\in L^2_{0,\sigma}(\rplus)$ (in the sense of Definition~\ref{def:leray_hopf}). Then, from
	\begin{align*}
		\begin{aligned}
                  [\nabla\bvn]\bvn &\in L^{s}(0,T;(L^{r}(\rplus))^3)&&\quad\text{ for some }s,r\in (1,\infty)\,,\\
                  \bvn_{0}&\in X_{s,r}
                  \,, \end{aligned}
	\end{align*}
	it follows that $\bvn\in W^{2,1}_{s,r}$ 
	and that there exists $\pi\in
        L^s(0,T;L^{r}_{loc}(\overline{\rplus}))$, with $\nabla\pi\in
        L^s(0,T;(L^r(\rplus))^3)$, such that  
 \begin{equation}
      \label{eq:Navier-Stokes_strong}
      \begin{aligned}
        \partial_{t}\bvn-\nu\Delta\bvn+[\nabla\bvn]\bvn+\nabla \pi
        &=\mathbf{0}&&\quad\text{ a.e.\ in }(0,T)\times \rplus \,,
        \\
        \dive\bvn&=0&&\quad \text{ a.e.\ in }(0,T)\times \rplus \,,
        \\
        \bvn&=\mathbf{0}&&\quad \text{ a.e.\ on
        }(0,T)\times\partial\rplus\,,
        \\
        \bvn(0)&=\bvn_0 &&\quad \text{ a.e.\ in }\rplus \,.
      \end{aligned}
    \end{equation}
	In addition, it holds that
	\begin{align}\label{eq:Navier-Stokes_strong_estimate}
        \begin{aligned}
          \int_{0}^{T}\big\{\|\partial_t\bvn \|_{L^r(\rplus)}^{s}
            +\|\bvn\|_{W^{2,r}(\rplus)}^{s}+\|\nabla
            \pi\|_{L^r(\rplus)}^{s}\big\}\,\de\tau
          \leq c(r,s,T,\nu)\,|||\bvn_{0}|||_{s,r}^{s}\,.
             \end{aligned}
	\end{align}
\end{theorem}
\begin{remark}
  The above result holds for any domain for which we have the
  parabolic maximal  regularity result
  from Theorem~\ref{thm:linear-Stokes} and the validity of the Helmholtz decomposition,
  e.g., beside the full and half-space, it is also valid for smooth bounded and exterior domains. In particular, for a  bounded $C^2$-domain $\Omega$, the proof can be greatly simplified since any bound in $L^r(\Omega)$ implies also corresponding bounds in $L^{p}(\Omega)$ for all $1\leq p<r$, making the steps of the duality argument easier to be rigorously justified. 
\end{remark}
\begin{remark}\label{rem:bochner_derivative}
  Resorting to \cite[Proposition 2.6.1]{Droniou2023}, for every $s,r\in (1,\infty)$, we find that the inclusion
  \begin{align*}
  	W^{2,1}_{s,r}\subseteq W^{1,s}(0,T;(L^r(\rplus))^3)\,,
  \end{align*}
  applies, and it holds 
  \begin{align*}
  	\frac{\de\bvn}{\de t}=\partial_{t}\bvn\quad\text{  in }L^s(0,T;(L^r(\rplus))^3)\,.
  \end{align*}
  Moreover, since in the definition of Leray-Hopf (weak) solution we have
  the quantity $\frac{\de_{e}\bvn}{\de t}$ we observe that they coincide
  also with this, at least in the sense of divergence-free distributions, that
  is
\begin{equation*}
        \int_0^T\bigg\langle \frac{\de\bvn}{\de
    t},\bphi\bigg\rangle\,\de\tau=      \int_0^T\bigg\langle \frac{\de_e\bvn}{\de
    t},\bphi\bigg\rangle\,\de\tau=-\int_{0}^{T}\int_{\rplus}\left\{\nu\nabla\bvn:\nabla\bphi+
  [\nabla\bvn]\bvn\cdot\bphi\right\}\dx\de \tau\,,
\end{equation*}
valid  for $\bphi\in L^{\max\{4,s'\}}(0,T;C^{\infty}_{0,\sigma}(\rplus))$. 
%
\end{remark}
  \begin{proof}[Proof (of Theorem~\ref{thm:giga})]
  	To begin with, let $(\bvpsi,\lambda)^\top\in
        W^{2,1}_{s,r}\times L^s(0,T;W^{1,r}(\rplus))$,  be a strong
        solution of the unsteady Stokes equations \eqref{eq:Stokes}
        with right-hand side $\mathcal{P}\mathbf{f}\coloneqq \mathcal{P}\big[-[\nabla\bvn]\bvn\big]
        \in L^{s}(0,T;L^{r}_{0,\sigma}(\rplus))$ and initial datum $\bvn_0\in
        X_{s,r}$. Note that since $[\nabla\bvn]\bvn$ is not
        divergence-free, we need to apply the Helmholtz projection
        operator ${\mathcal{P}:\ (L^{p}(\rplus))^{3}\to
        L^{p}_{0,\sigma}(\rplus)}$, for $p \in (1,\infty)$, which is defined in
        $(L^{p}(\rplus))^{3}$, $p \in (1,\infty)$, as well as in space of more regular
        functions, by solving a weak Neumann problem, cf.~\cite{Gal2000a}. Note  also that
        $\mathcal{P}=\text{id}$ in both $L^p_{0,\sigma}(\rplus)$ and $W^{1,p}_{0,\sigma}(\rplus)$~for~all~${p\in (1,\infty)}$.

        The existence of such $(\bvpsi,\lambda)^\top$ is provided by Theorem \ref{thm:linear-Stokes}, \textit{i.e.}, it holds that 
  	\begin{equation}\label{thm:giga.1}
  		\begin{aligned}
  			\partial_{t}\bvpsi-\nu\Delta\bvpsi+\nabla \lambda
  			&=-\mathcal{P}\big[[\nabla\bvn]\bvn\big]&&\quad\text{ a.e.\ in }(0,T)\times \rplus \,,
  			\\
  			\dive\bvpsi&=0&&\quad \text{ a.e.\ in }(0,T)\times \rplus \,,
  			\\
  			\bvpsi&=\mathbf{0}&&\quad \text{ a.e.\ on
  			}(0,T)\times\partial\rplus\,,
  			\\
  			\bvpsi(0)&=\bvn_0 &&\quad \text{ a.e.\ in }\rplus \,.
  		\end{aligned}
  	\end{equation}
	Next, let  $\bphi\in \bigcap_{q,p\in (1,\infty)}{{\mathaccent23 W^{2,1}_{q,p}}}$ be fixed, but arbitrary,
	where for every $q,p\in (1,\infty)$, we define
	\begin{align}
		{\mathaccent23 W^{2,1}_{q,p}}\coloneqq \left\{\bphi\in W^{2,1}_{q,p}\ \Bigg|\;
		\begin{aligned}
		&\bphi(T)=\mathbf{0}\text{ a.e.\ in }\rplus\,,\\& \bphi(\tau)\in W^{1,p}_{0,\sigma}(\rplus)\text{ for all }\tau\in (0,T)\end{aligned}\right\}\,.
	\end{align}

        On the one hand, if we integrate the inner product of \eqref{thm:giga.1}$_1$ and $\bphi$ over $(0,T)\times \rplus$, we find that
  	\begin{equation*}
  		\int_0^T\int_{\rplus} \bvpsi\cdot(\partial_t\bphi+\nu\Delta\bphi)\dx \de\tau=-\int_{\rplus}
  		\bvn_0\cdot\bphi(0)\dx+\int_0^T\int_{\rplus}[\nabla\bvn]\,\bvn\cdot\bphi\dx \de\tau\,,
  	\end{equation*}
  	where we used the integration-by-parts formula
        \cite[Proposition 2.5.2]{Droniou2023} (which is allowed in
        view of $\bvpsi\in W^{1,s}(0,T;(L^r(\rplus))^3)$ and
        $\bphi\in W^{1,s'}(0,T;(L^{r'}(\rplus))^3)$ and Remark
        \ref{rem:bochner_derivative}) in time, Green's second identity
        for Sobolev functions in space for a.e.\ $\tau\in (0,T)$
        (which is allowed since $\bvpsi(\tau)\in (W^{2,r}(\rplus))^3$
        and $\bphi(\tau)\in (W^{2,r'}(\rplus))^3$
        for~a.e.~$\tau\in (0,T)$), and $\bphi(T)=\mathbf{0}$ a.e.\ in
        $\rplus$ as well as $\bvpsi=\bphi=\mathbf{0}$~a.e.~on~$(0,T)\times\partial\rplus$. Moreover, we used  that, by
        the properties of the Helmholtz decomposition~(cf.~\cite{Soh2001}) 
        \begin{equation*}
        \int_{\rplus}\mathcal{P}\big[[\nabla\bvn]\,\bvn\big]\cdot\bphi\dx=\int_{\rplus}[\nabla\bvn]\,\bvn\cdot\mathcal{P}\bphi\dx=\int_{\rplus}[\nabla\bvn]\,\bvn\cdot\bphi\dx\,.
      \end{equation*}
  	
  	On the other hand, if we test the weak formulation in Remark
        \ref{rem:equiv_form} with $\bphi$  (which is allowed since in
        particular $\bphi\in W^{1,4,2}(0,T;W^{1,2}_{0,\sigma}(\rplus),(W^{1,2}_{0,\sigma}(\rplus))^*)$), we~find~that
  	\begin{equation*}
  	\int_0^T\int_{\rplus} \bvn\cdot(\partial_t\bphi+\nu\Delta\bphi)\dx \de\tau=-\int_{\rplus}
  	\bvn_0\cdot\bphi(0)\dx+\int_0^T\int_{\rplus}[\nabla\bvn]\,\bvn\cdot\bphi\dx \de\tau\,,
  \end{equation*}
  where we used that $\frac{\mathrm{d}_e\bphi}{\mathrm{d}t}=e(\partial_t\bphi)$ in $L^2(0,T;(W^{1,2}_{0,\sigma}(\rplus))^*)$, 
  Green's second identity for Sobolev functions in space for a.e.\ $\tau\in (0,T)$ (which is allowed since $\bvn(\tau)\in (W^{1,2}(\rplus))^3$ and $\bphi(\tau)\in (W^{2,2}(\rplus))^3$ for a.e.\ $\tau\in (0,T)$), and due to  $\bphi(T)=\mathbf{0}$ a.e.\ in $\rplus$ as well as  $\bvn=\mathbf{0}$ a.e.\ on $(0,T)\times\partial\rplus$.
  
  In summary, for every
  $\bphi\in \bigcap_{q,p\in (1,\infty)}{{\mathaccent23 W^{2,1}_{q,p}}}$, we have that
  \begin{align}\label{thm:giga.4} 
    \int_0^T\int_{\rplus} ( \bvn-\bvpsi)\cdot(\partial_t\bphi+\nu\Delta\bphi)\dx \de\tau=0\,.
  \end{align} 
  %
	Next, let $\mathbf{\Phi}\in (C^\infty_0((0,T)\times\rplus))^3$ be fixed, but arbitrary. Then, let
$(\widetilde{\bphi}, \widetilde{\eta} )^\top\in \cap_{q,p\in (1,\infty)}{W^{2,1}_{q,p}\times L^q(0,T;W^{1,p}(\rplus))}$ be a strong solution to 
 the unsteady Stokes equations \eqref{eq:Stokes} with right-hand side $\mathbf{f}\coloneqq
 \mathcal{P}\big[-\mathbf{\Phi}(T-\cdot)\big]$ belonging to the space $
 \bigcap_{q,p\in (1,\infty)}
  L^q(0,T;L_{0,\sigma}^p(\rplus))$ and initial datum $\bvpsi_0\coloneqq \mathbf{0}$ provided by Theorem~\ref{thm:linear-Stokes}, \textit{i.e.}, it holds that
\begin{equation*}
  \begin{aligned}
    \partial_{t}\widetilde{\bphi}-\nu\Delta\widetilde{\bphi}+\nabla
    \widetilde{\eta} &= \mathcal{P}\big[-\mathbf{\Phi}(T-\cdot)\big]&&\quad\text{ a.e.\ in }(0,T)\times \rplus \,,
    \\
    \dive\widetilde{\bphi}&=0&&\quad \text{ a.e.\ in }(0,T)\times \rplus \,,
    \\
    \widetilde{\bphi}&=\mathbf{0}&&\quad \text{ a.e.\ on }(0,T)\times\partial\rplus\,,
    \\
    \widetilde{\bphi}(0)&=\mathbf{0} &&\quad \text{ a.e.\ in }\rplus \,.
  \end{aligned}
\end{equation*}
Then,
$(\bphi,\eta)^\top\coloneqq
(\widetilde{\bphi}(T-\cdot),\widetilde{\eta}(T-\cdot))^\top\in
\bigcap_{q,p\in (1,\infty)}{{\mathaccent23 W}^{2,1}_{q,p}\times
  L^q(0,T;W^{1,p}(\rplus))}$~solves the backward unsteady Stokes
equations with the divergence-free right-hand side
$ \mathcal{P}[-\mathbf{\Phi}]$ 
and final datum $\bvpsi_{T}\coloneqq \mathbf{0}$, \textit{i.e.}, it
holds that
\begin{equation}
	 \label{eq:Stokes-backward}
	\begin{aligned}
		-\partial_{t}\bphi-\nu\Delta\bphi+\nabla \eta &=\mathcal{P}\big[-\boldsymbol{\Phi}\big]&&\quad\text{ a.e.\ in }(0,T)\times \rplus \,,
		\\
		\dive\bphi&=0&&\quad \text{ a.e.\ in }(0,T)\times \rplus \,,
		\\
		\bphi&=\mathbf{0}&&\quad \text{ a.e.\ on }(0,T)\times\partial\rplus\,,
		\\
		\bphi(T)&=\mathbf{0} &&\quad \text{ a.e.\ in }\rplus \,.
	\end{aligned}
\end{equation}
Since $\eta\in L^q(0,T;W^{1,p}(\rplus))$ for all $p,q \in (1,\infty)$,  $\bvn\in L^2(0,T;W^{1,2}_{0,\sigma}(\rplus))$, and $\bvpsi\in
L^s(0,T;W^{1,r}_{0,\sigma}(\rplus))$, we get, using integration by parts,
\begin{align}\label{eq:0}
	  \int_0^T\int_{\rplus}\nabla \eta\cdot \bvn\dx\,
  \de\tau=\int_0^T\int_{\rplus}\nabla \eta\cdot \bvpsi\dx\,
  \de\tau=0\,. 
\end{align}
Thus,
for~every~${\mathbf{\Phi}\in (C_0^\infty((0,T)\times \rplus))^3}$ we
obtain from \eqref{thm:giga.4}--\eqref{eq:0}, also using
$\mathcal{P}\bvn=\bvn$ and $\mathcal{P}\bvpsi=\bvpsi$, that 
\begin{equation*}
  \begin{aligned}
    0 &= \int_0^T\int_{\rplus}(
    \bvn-\bvpsi)\cdot(\partial_{t}\bphi+\nu\Delta\bphi)\,\dx \,\de\tau
    \\
    & =\int_0^T\int_{\rplus}(
    \bvn-\bvpsi)\cdot\mathcal{P}\mathbf{\Phi}\dx\, \de\tau=
    \int_0^T\int_{\rplus}\mathcal{P}(
    \bvn-\bvpsi)\cdot\mathbf{\Phi}\dx\, \de\tau
    \\
    & = \int_0^T\int_{\rplus}( \bvn-\bvpsi)\cdot\mathbf{\Phi}\dx\,
    \de\tau\,.
  \end{aligned}
\end{equation*}
Thus, the fundamental theorem in the calculus of variations
yields  ${\bvn=\boldsymbol{\psi}}$ a.e.\ in $(0,T)\times \rplus$.
This identification implies that $\bvn\in W^{2,1}_{s,r}$.  Note that
$[\nabla\bvn]\,\bvn-\mathcal{P}\big[[\nabla\bvn]\,\bvn\big]=\nabla
\lambda_{\bvn}$, for some
$\lambda_{\bvn}\in L^{s}(0,T;L^{r}_{loc}(\overline{\rplus}))$ with
 $\nabla \lambda_{\bvn}\!\in\!
 L^{s}(0,T;(L^{r}(\rplus)^{3})$. Consequently, we have in 
~\eqref{eq:Navier-Stokes_strong} that
$\pi=\lambda+\lambda_{\bvn}$, explaining why the pressure (contrary to
its gradient) belongs to a local space, ending the proof.
\end{proof}

In order to apply Theorem~\ref{thm:giga}, we need to estimate the convective
term~in~(mixed) Lebesgue spaces.
\begin{proposition}
	\label{prop:prop1}
        Let  \hspace{-0.1mm}$\bvn_0\hspace{-0.175em}\in\hspace{-0.175em} L^2_{0,\sigma}(\rplus)$ \hspace{-0.1mm}and
        \hspace{-0.1mm}$\bvn\hspace{-0.175em}\in\hspace{-0.175em}
        L^2(0,T;W^{1,2}_{0,\sigma}(\rplus))\cap
        L^\infty(0,T;L_{0,\sigma}^2(\rplus))$ be  a \hspace{-0.15mm}corresponding \hspace{-0.15mm}Leray--Hopf \hspace{-0.15mm}(weak)
  solution \hspace{-0.1mm}(in \hspace{-0.15mm}the
  \hspace{-0.15mm}sense \hspace{-0.15mm}of
  \hspace{-0.15mm}Definition~\hspace{-0.15mm}\ref{def:leray_hopf}).~\hspace{-0.5mm}Then,
  \hspace{-0.15mm}from
	\begin{equation}
		\label{eq:reg-time-d}
		\bvn\in L^\beta(0,T;(L^{\infty}(\rplus ))^3)\quad \text{ for some
		}\beta\in( 1,2)\,,
	\end{equation}
	%
        for every $r\in (1,\frac{4}{4-\beta})$, it follows that
	\begin{equation}
		\label{eq:reg-convective}
		\begin{aligned}
                  & [\nabla\bvn]\bvn\in L^{\frac{2r\beta}{4r+r \beta-4}}(0,T;(L^{r}(\rplus))^3) \,.
                \end{aligned}
              \end{equation}
\end{proposition}
\begin{proof} 

  For $r\in (1,2)$ we find, by H\"older inequality, convex interpolation, and the energy
  inequality \eqref{def:leray_hopf.2}, for a.e.\ $t\in (0,T)$,   that 
  \begin{align*}
    \|([\nabla\bvn]\bvn)(t)\|_{L^r(\rplus )}&\leq
    \|\bvn(t)\|_{L^2(\rplus )}^{\frac{2}{r}-1}\|\bvn(t)\|_{L^\infty(\rplus )}^{2(1-\frac{1}{r})}\|\nabla\bvn(t)\|_{L^2(\rplus )}
\\&\leq\|\bvn_0\|_{L^2(\rplus )}^{\frac{2}{r}-1}
  \|\bvn(t)\|_{L^\infty(\rplus)}^{2(1-\frac{1}{r})}\,\|\nabla\bvn(t)\|_{L^2(\rplus )} \,, 
  \end{align*}
Thus, using again H\"older inequality we get 
   \begin{equation*}
    \int_0^T\hspace*{-1.5mm}\|([\nabla\bvn]\bvn)(t)\|_{L^{r}(\rplus)}^\alpha \, \de t \le
    c\, \bigg (\int_0^T\hspace*{-1.5mm}\|\nabla\bvn(t)\|_{L^{2}(\rplus)}^2 \, \de t
    \bigg)^{\frac r2} \!\bigg (\int_0^T\hspace*{-1.5mm}\|\bvn(t)\|_{L^{\infty}(\rplus)}^{\frac{4\alpha}{2-\alpha}\frac{r-1}{r}} \, \de t
    \bigg)^{\frac {2-r}2} ,
  \end{equation*} 
  which is finite, due to \eqref{eq:reg-time-d} and
  $\bvn \in L^2(0,T;W^{1,2}_{0,\sigma}(\rplus))$, provided
  $\alpha \le \frac{2r\beta}{4r+r \beta-4}$. Requiring $\alpha>1$ yields
  $r\in (1,\frac{4}{4-\beta})$. This proves that
  \eqref{eq:reg-convective} holds for~all~${r\in (1,\frac{4}{4-\beta})}$.
\end{proof}
Given the assumptions of Theorem~\ref{thm:main}, a combination of Theorem~\ref{thm:giga}
and Proposition~\ref{prop:prop1} yields the following alternative weak formulation
to~\eqref{def:leray_hopf.1}.\enlargethispage{7.5mm}
%
\begin{corollary}
  \label{cor:alternative_formulation}
  Let  $\bvn_0\in L^2_{0,\sigma}(\rplus)$ and 
  $\bvn\in L^2(0,T;W^{1,2}_{0,\sigma}(\rplus))\cap
  L^\infty(0,T;L_{0,\sigma}^2(\rplus))$ be  a \hspace{-0.15mm}corresponding \hspace{-0.15mm}Leray--Hopf \hspace{-0.15mm}(weak)
  solution \hspace{-0.1mm}(in \hspace{-0.15mm}the
  \hspace{-0.15mm}sense \hspace{-0.15mm}of
  \hspace{-0.15mm}Definition~\hspace{-0.15mm}\ref{def:leray_hopf}).~\hspace{-0.5mm}Then,
  \hspace{-0.15mm}from
  \begin{align*}
    \begin{aligned}
      \bvn&\in L^\beta(0,T;(L^{\infty}(\rplus ))^3)&&\quad \text{ for some
      }\beta\in( 1,2)\,,
      \\
      \bvn_{0}&\in X_{s,r}
      &&\quad\text{ for some }r\in
      (1,\tfrac{4}{4-\beta})\text{ and } s:=\frac{2r\beta}{4r+r \beta-4}\,,
    \end{aligned}
  \end{align*}
  it follows 
  that
  $\bvn\in W^{2,1}_{s,r}$ 
  and that there exists $\pi\in L^s(0,T;(W^{1,r}_{loc }(\overline{\rplus}))^3)$, with $\nabla\pi\in L^s(0,T;(L^r(\rplus))^3)$, such that \eqref{eq:Navier-Stokes_strong} and \eqref{eq:Navier-Stokes_strong_estimate} apply.
\end{corollary}
\begin{remark}
  We observe that the regularity obtained in 
  Corollary~\ref{cor:alternative_formulation} is not enough to infer directly that
  $\partial_t\bvn\cdot\bvn\in L^1((0,T)\times\rplus)$, which will allow us to justify (after smoothing) that the integral $\int_{0}^{T}\int_{\rplus}\partial_t\bvn \cdot\bvn\dx\,\de\tau$ converges
  to one-half of the difference of squared $L^2$-norms. For instance, observe that
  Corollary~\ref{cor:alternative_formulation} and interpolation between
  $L^{\infty}(0,T;(L^{2}(\rplus))^3)$ and $L^{\beta}(0,T;(L^{\infty}(\rplus))^3)$ imply that
\begin{align*}
  \partial_t\bvn&\in
  \big(L^{\smash{\frac{2(2+\beta)}{4+\beta}}}((0,T)\times\rplus)\big)^{3}\,,
  \\ 
  \bvn&\in
  \big(L^{2+\beta}((0,T)\times\rplus)\big)^{3}\,, 
\end{align*}
and that only if $\beta\geq2$, we have that
$$\frac{4+\beta}{2(2+\beta)}+\frac{1}{2+\beta}\leq1\,,$$ 
which is needed to apply H\"older's
inequality. The case $\beta\geq2$, however, 
 is excluded since ``large'' $\beta$ will mean $\bvn$ being already 
in a scaling-invariant class of regularity. This explains also why the treatment of the
``linear'' term involving the time derivative is again possible only employing special
properties of the approximation, deduced from the H\"older continuity assumption, see 
Section~\ref{sec:NSE}.
\end{remark}
\section{Mollification of divergence-free functions in H\"older spaces}
\label{sec:mollification}
To handle the difficulties coming from the (homogeneous) Dirichlet boundary condition, we adapt and modify a
 strategy as in~\cite{CET1994,DR2000,Eyi1994} to justify the calculations. To
understand the critical technical point, 
observe that a traditional --probably the unique-- way to take decisive advantage from the
additional H\"older continuity of the solution is to use properties of mollification
operators.  For divergence-free functions, this is particularly delicate in the presence
of a solid boundary, see Remark~\ref{rem:smoothing}. %
%
%
%
%

For this purpose, we recall some basic results, which are independent
of the incompressible case. We choose a radially symmetric
mollification kernel $\rho\in C^\infty_0(\R^{d})$ such that
\begin{align*}
  \rho\ge 0\quad\text{ in }\R^{d}\,,\qquad \textup{supp}\,\rho\subseteq
  B_1^{d}(0)\,,\qquad \int_{\R^{d}}{\rho(x)\,\mathrm{d}x}=1\,, 
\end{align*}
and define the family of Friedrichs mollifiers $(\rho_\varepsilon)_{\varepsilon\in (0,1]}\subseteq C^\infty_0(\mathbbm{R}^d)$ for every $\varepsilon\in (0,1]$ and $x\in  \mathbbm{R}^d$ by $\rho_\varepsilon(x)\coloneqq\frac{1}{\varepsilon^d}\rho(\frac{x}{\varepsilon})$. Then, for every function $u\in L^1_{\textup{loc}}(\mathbbm{R}^d)$,~we~\mbox{define}, by using standard convolution, for every $\varepsilon\in (0,1]$ and $x\in \mathbbm{R}^d$
	\begin{align}\label{def:molli1}
        \begin{aligned}
		   (\rho_\varepsilon\ast u)(x)\coloneqq \int_{\mathbbm{R}^d}{\rho_\varepsilon(x-y)u(y)\,\mathrm{d}y}=\int_{\mathbbm{R}^d}{\rho_\varepsilon(y)u(x-y)\,\mathrm{d}y}\,.
        \end{aligned}
	\end{align}
	In fact, the second last integral is evaluated only in $B_\varepsilon^d(x)$, while the  last integral is evaluated only in $B_\varepsilon^d(0)$. So, if needed, we can restate  \eqref{def:molli1}, for every $\varepsilon\in (0,1]$ and $x\in \mathbbm{R}^d$,~as
	\begin{align}\label{def:molli2}
        \begin{aligned} 
		  (\rho_\varepsilon\ast u)(x)=\int_{B_\varepsilon^d(x)}{\rho_\varepsilon(x-y)u(y)\,\mathrm{d}y}=\int_{B_\varepsilon^d(0)}{\rho_\varepsilon(y)u(x-y)\,\mathrm{d}y}\,.
        \end{aligned}
	\end{align}
 Denote by
$\overline{(\cdot)}\colon L^1_{\textup{loc}}(\rplusd)\to
L^1_{\textup{loc}}(\mathbbm{R}^d)$ the zero extension operator outside of
$\rplusd$, for~every $u\in L^1_{\textup{loc}}(\rplusd)$ defined by
\begin{align}
  \label{eq:zero-extension}
  \overline{u}\coloneqq
  \begin{cases}
    u&\text{ in }\rplusd\,,
    \\
    0&\text{ in }\mathbbm{R}^d\setminus \rplusd\,.
  \end{cases}
\end{align}
Then, for every $u\in L^1_{\textup{loc}}(\rplusd)$, for
every $\varepsilon \in (0,1]$ and $x\in\mathbbm{R}^d$, we abbreviate
\begin{align*}
	u_\varepsilon(x)\coloneqq (\rho_{\varepsilon} \ast \overline{u})(x)\,.
\end{align*}
Note that if
$u\in  {\dot C}^{0,\alpha}(\overline{\rplusd})\cap L^{1}_{\textup{loc}}(\rplusd)$,
where $\alpha\in  (0,1]$, with $u=0$~on~$\partial\rplusd$, then
$\overline{u}\in \dot{C}^{0,\alpha}(\mathbbm{R}^d)\cap
L^1_{\textup{loc}}(\mathbbm{R}^d)$ with
$[\overline{u}]_{\dot{C}^{0,\alpha}(\mathbbm{R}^d)}=[u]_{{\dot C}^{0,\alpha}(\overline{\rplusd})}$.~Similarly, if ${u\in C^{0,\alpha}(\overline{\rplusd})}$, where $\alpha\hspace{-0.15em}\in\hspace{-0.15em} (0,1]$, with $u\hspace{-0.15em}=\hspace{-0.15em}0$ on
$\partial\rplusd$, then $\overline{u}\hspace{-0.15em}\in\hspace{-0.15em} C^{0,\alpha}(\mathbbm{R}^d)$ with ${\|\overline{u}\|_{C^{0,\alpha}(\mathbbm{R}^d)}\hspace{-0.15em}=\hspace{-0.15em}\|u\|_{C^{0,\alpha}(\overline{\rplusd})}}$.

The basic calculus estimates we can employ are the following.
\begin{lemma}
  \label{lem:convolution-Holder0}
  Let $d\in \mathbb{N}$ and $\bu\in {\dot C}^{0,\alpha}(\overline{\rplusd})$, where $\alpha\in (0,1)$. Then, for every $\vep>0$, it holds that
  \begin{align}
    \label{eq:conv20}
    &   \|\bu_{\vep}-\bu\|_{L^\infty(\rplusd)}\leq
                            [u]_{{\dot C}^{0,\alpha}(\overline{\rplusd})}\, \vep^{\alpha}\,, 
    \\
     \label{eq:conv30}
    &\|\nabla\bu_{\vep}\|_{L^\infty(\rplusd)}\leq c_\rho\,
                            [u]_{{\dot C}^{0,\alpha}(\overline{\rplusd})}\,\vep^{\alpha-1}\,, 
  \end{align}
  where $c_\rho\coloneqq \int_{\R^{d}}|\nabla\rho|\dx>0$.
\end{lemma}
\begin{proof}
  	See~\cite{CET1994,Ber2023b}.
      \end{proof}
        To pass to our vector-valued problem, we first observe that the zero
  extension \eqref{eq:zero-extension} is still applicable component-wise. What is relevant
  is that if ${\mathbf{u}\in W^{1,1}_{0,\sigma}(\rplusd)}$, then 
  $\overline{\mathbf{u}}\in W^{1,1}_{0,\sigma}(\mathbbm{R}^d)$.
      \begin{remark}
        \label{rem:smoothing}
        The direct approach for smoothing by convolution cannot be applied
        to our problem:
          
        On the one hand, if $\bun \in W^{1,1}_{0,\sigma}(\rplusd)$, then, by the commutation properties
        of the convolution operator, it holds that $\bun_\vep\coloneqq\rho_{\vep} \ast \overline{\bun} \in (C^\infty(\mathbbm{R}^d))^d\cap (W^{1,1}(\mathbbm{R}^d))^d$ with
     \begin{align*}
           \left. \begin{aligned} 
          \textup{div}\,\bun_\vep&=\rho_{\vep}
          \ast (\textup{div}\,\overline{\bun})\\&=\rho_{\vep} \ast(
          \overline{\textup{div}\,\bun})=0
          \end{aligned}\quad \right\}\quad\text{ a.e.\ in }\R^{d}\,,
        \end{align*}
        which aligns with divergence-free requirement for test functions.

    On the other hand, upon studying the support, in general, it holds that
    \begin{align*}
    	\textup{supp}\,\bun_\vep\not\subseteq \rplusd\,,
    \end{align*}
    violating the support requirement for test functions.

    \medskip
    
    Alternatively, if we localize the vector field $\mathbf{u}\in W^{1,1}_{0,\sigma}(\rplusd)$ in the interior~of~$\rplus$ using a bump
    function $\phi\in C^\infty_{c}(\rplusd)$ and choose $\vep>0$  smaller than the distance between
    $\textup{supp}\,\phi$~and~$\partial\rplus$, we can produce a vector field that has a
    compact support~in~$\rplusd$, but that is not divergence-free. Applying the Leray projection, 
    i.e., $\mathcal{P}(\rho_{\vep}*(\overline{\bun}\,\phi))$, would yield a valid
    test function. However, since $\mathcal{P}$ is a non-local operator, we would lose the crucial
    cancellation properties associated with the convective term.
    
  \end{remark}
  We then adapt the standard mollification to our problem with Dirichlet boundary
  conditions on a hyperplane and divergence-free constraint. To this end, we use a
  smoothing obtained by a combination of \textit{convolution} in all variables, and
  \textit{translation} in the $x_d$-direction.
  \begin{definition}[Convolution-translation operator]
    \label{def:uepsilonprime}
    Let $d,\ell\hspace{-0.175em}\in \hspace{-0.175em}\mathbb{N}$ and $\bun\hspace{-0.175em}\in\hspace{-0.175em} (L^{1}_{\textup{loc}}(\rplusd))^{\ell}$.  Then, denoting
    by $\overline{\mathbf{u}}\in (L^1_{\textup{loc}}(\mathbbm{R}^{d}))^{\ell}$ the zero extension defined
    component-wise~by~\eqref{eq:zero-extension} 
    for every $\vep>0$, we define 
    $S_\varepsilon(\mathbf{u})\colon\R^{d}\to \mathbbm{R}^{\ell}$, for every
    $x\in \mathbbm{R}^d$, by
    \begin{equation*}
    \begin{aligned}
      S_\varepsilon(\mathbf{u})(x)&\coloneqq \int_{B_{\vep}^{d}(0)}\rho_{\vep}(y)\,\overline{\bun}(x-2\vep\,
      \mathrm{e}_d-y)\dy
      \\
      &=\big[\rho_{\vep}*(\tau_{2\vep}\overline{\bun})\big](x)\,,
    \end{aligned}
  \end{equation*}
  where $\tau_{2\vep}\colon (L^1_{\textup{loc}}(\mathbbm{R}^{d}))^{\ell}\to (L^1_{\textup{loc}}(\mathbbm{R}^{d}))^{\ell}$ denotes the translation in the direction
  $-\mathrm{e}_d$ with increment $2\vep$,  for every $\mathbf{w}\in (L^1_{\textup{loc}}(\mathbbm{R}^{d}))^{\ell}$ defined by 
  $(\tau_{2\vep}\mathbf{w}) (x)\coloneqq \mathbf{w}(x-2 \vep \,\mathrm{e}_d)$ for a.e.\ $x\in \mathbbm{R}^{d}$.
\end{definition}
A similar approach has previously been used in the mathematical analysis~of~shells (\textit{i.e.}, without constraints on the divergence) and for general domains satisfying the cone
property; see for instance, Blouza and Le Dret~\cite{BLD2001} and the references~therein.
Here, it is crucial that the domain is flat and the translation is the same for~all~points. This allows us to avoid localizations (\textit{cf}.~Remark~\ref{rem:smoothing}) and permits to
preserve the divergence-free constraint.  Another ``easy case'' to which a similar
approach could be adapted is the case of a domain that is star-shaped with respect to a ball. A completely different approach needs to be used in the case of a general (bounded) smooth domain,
since neither translation in a single direction is possible nor dilation towards a
single center, see~\cite{BKR2024b}. Note that the result in this paper are simpler than those in~\cite{BKR2024b} for
what concerns the smoothing and the computations are particularly neat, while other
technical problems arise due to the fact that the domain is unbounded.\enlargethispage{2mm}

\medskip

The mollification operator from Definition~\ref{def:uepsilonprime} shares many of the
usual properties of the standard mollification in the case without boundaries. The main
properties we will use are summarized in the next two propositions: the first one concerns
simultaneous preservation of zero boundary conditions and incompressibility and the
regularity and order of convergence as $\varepsilon\to0^{+}$ in Lebesgue spaces. The
second one concerns the properties in case of a further application of a standard
convolution by a Friederichs mollifier.

\begin{proposition}
  \label{prop:basic_properties}
  Let $d,\ell\in \mathbb{N}$. Then, for every $\bun \in (L^1_{\textup{loc}}(\rplusd))^{\ell}$, the
  following statements apply:
  \begin{itemize}[noitemsep,topsep=2pt,leftmargin=22pt,labelwidth=\widthof{(iii)}]
  \item[(i)] For every $\vep>0$, it holds that $S_{\vep}(\bun) \in (C^\infty(\rplusd))^{\ell}$ with
   $$\mathrm{supp}(S_{\vep}(\bun) )\subseteq \vep\, \mathrm{e}_d + \rplusd\,.$$
  \item[(ii)]  If $\ell=d$ and  $\bun\in (W^{1,1}_{0,\sigma}(\rplusd)
    )^\ell$, then, for every $\vep>0$, it holds that $S_{\vep}(\bun)
    \in (C^{\infty}_{\sigma}(\rplusd) )^\ell $ and
    $S_{\vep}(\bun)(x_{1},\dots,x_{d-1},0)=\mathbf{0}$.
  \item[(iii)] There exists a constant $c_d>0$, depending only on $d$, such that for every
    $ \vep >0$, it holds that
    \begin{alignat}{2}
      \vert S_{\vep}(\bun) \vert &\leq c_d
      \,M(\vert\overline{\bun}\vert)&&\quad\text{ a.e.\ in }\rplusd
      \,,\label{prop:basic_properties.1}
      \\
      \intertext{and if
        $\bun\in (W^{1,1}_0(\rplusd))^{\ell}$,
        then, for every
    $ \vep >0$, it holds that}\vert\nabla S_{\vep}(\bun) \vert &\leq c_d
      \,M(\vert\nabla\overline{\bun}\vert) &&\quad\text{ a.e.\ in }\rplusd
      \,,\label{prop:basic_properties.2}
      \\
      \vert S_{\vep}(\bun) -\bun\vert &\leq c_d \,\vep
      \,\big(M(\vert\overline{\bun}\vert)+M(\vert\nabla\overline{\bun}\vert)\big)
      &&\quad\text{
        a.e.\ in }\rplusd \,,\label{prop:basic_properties.2.1}
    \end{alignat}
    where $M\colon L^1_{\textup{loc}}(\mathbbm{R}^d)\to L^0(\mathbbm{R}^d)$ denotes the Hardy--Littlewood~maximal~\mbox{operator}.
    
  \item[(iv)] If $\bun\in (L^r(\rplusd))^{\ell}$ for $r\in (1,\infty)$, then there
    exists a constant $c_r>0$, depending only on $d$ and $r$, such that for
    every $ \vep >0$, it holds that 
    \begin{align}
      \|S_{\vep}(\bun) \|_{L^r(\rplusd)} &\leq c_r\,
                                         \|\bun\|_{L^r(\rplusd)}\,,
                                         \label{prop:basic_properties.3}
      \\
      \intertext{and if $\bun\in (W^{1,r}_0(\rplusd))^{\ell}$ for $r\in (1,\infty)$,
      then, for every $\vep>0$, it holds that}\|\nabla S_{\vep}(\bun) \|_{L^r(\rplusd)} &\leq
                                                    c_r\,\|\nabla\bun\|_{L^r(\rplusd)}\,,
                                                    \label{prop:basic_properties.4}
      \\
      \|S_{\vep}(\bun) -\bun\|_{L^r(\rplusd)} &\leq c_r\,\vep
                                              \,\|\bun\|_{W^{1,r}(\rplusd)}\,.
                                              \label{prop:basic_properties.4.1} 
    \end{align}
    
  \item[(v)] If $\bun\in (W^{1,r}_0(\rplusd))^{\ell}$ for $r\in (1,\infty)$, then
    $S_{\vep}(\bun) \to \bun$ in $(W^{1,r}_0(\rplusd))^{\ell}$ ($\vep \to 0^{+}$).
  \end{itemize}
\end{proposition}

Since we will use as test function
$\bphi=\rho_{\varepsilon}\ast S_{\vep}(\bvn)$, which is obtained
by a further standard mollification of the velocity vector field smoothed by convolution-translation,
we also need the following result.
\begin{proposition}
  \label{prop:admissibility}
  Let $d,\ell\in \mathbb{N}$. Then, for every $\mathbf{u} \in (L^{1}_{\textup{loc}}(\rplusd))^{\ell}$,
  the following statements apply:
	\begin{itemize}[noitemsep,topsep=2pt,leftmargin=20pt,labelwidth=\widthof{(iii)}]
		\item[(i)] For every $\vep\hspace{-0.1em}>\hspace{-0.1em}0$, it holds that $\rho_{\vep}*S_{\vep}(\bun)\hspace{-0.1em}\in\hspace{-0.1em} (C^\infty(\rplusd))^{\ell}$ with
                  ${\mathrm{supp}(\rho_{\vep}\ast S_{\vep}(\bun))\hspace{-0.1em}\subseteq\hspace{-0.1em} \rplusd}$. 
		\item[(ii)] If $\ell=d$ and $\mathbf{u}\in (W^{1,1}_{0,\sigma}(\rplusd) )^\ell$, then, for every $\vep>0$, it holds that
                  $\rho_{\vep}\ast S_{\vep}(\bun)\in
                  (C^{\infty}_{\sigma}(\rplusd) )^\ell$ and $\rho_{\vep}\ast S_{\vep}(\bun)(x_{1},\dots,x_{d-1},0)=\mathbf{0}$. 
	\end{itemize}
\end{proposition}
The proof of the above two propositions is rather standard and will be given in the
Appendix~\ref{sec:appendix}, for the sake of completeness.
\medskip

Moreover, in the sequel it will be crucial to have rates of convergence and bounds when
the  convolution-translation operator is applied to H\"older~continuous~functions. For
this reason, we sketch the proof of the counterpart of Lemma~\ref{lem:convolution-Holder0}
applied~to~$S_{\vep}(\bun)$. These will be the main properties of convolution in H\"older
spaces we will use in the sequel to estimate the convective term.
\begin{lemma}
  \label{lem:convolution-Holder}
  Let $d,\ell\in \mathbbm{N}$ and $\bun\in ({\dot C}^{0,\alpha}(\overline{\rplusd}))^{\ell}$, where $\alpha\in (0,1]$, such
  that $\bun=\mathbf{0}$ on $\partial \rplusd$. Then, for every $\vep>0$, it holds that
    \begin{align}\label{eq:conv1p}  
         \|\tau_{2\vep}\bun(\cdot+\by)-\tau_{2\vep}\bun\|_{L^\infty(\rplusd)}&\leq
       |\by|^{\alpha}\, [\bun]_{{\dot C}^{0,\alpha}(\overline{\rplusd})}\quad\textup{ for all }y\in \rplusd\,,
      \\
      \label{eq:conv2p}        \|S_{\vep}(\bun)-\bun\|_{L^\infty(\rplusd)}&\leq
                               3^{\alpha}\,[\bun]_{{\dot C}^{0,\alpha}(\overline{\rplusd})}\, \vep^{\alpha}\,, 
      \\
      \label{eq:conv3p}      \|\nabla S_{\vep}(\bun)\|_{L^\infty(\rplusd)}&\leq c_\rho\,
                               [\bun]_{{\dot C}^{0,\alpha}(\overline{\rplusd})}\,\vep^{\alpha-1}\,, 
    \end{align}
    where $c_\rho\coloneqq \int_{\R^{d}}|\nabla\rho|\dx>0$.
  \end{lemma}
  \begin{proof} 

    \textit{ad \eqref{eq:conv1p}.}
    The first estimate~\eqref{eq:conv1p} is just the statement of H\"older continuity
    since the translation is the same for all points in $\rplusd$.\enlargethispage{10mm}
  
    \textit{ad \eqref{eq:conv2p}.} Using that
    $\int_{B_{\vep}^{d}(0)}\rho_{\vep}\,\de \by=\int_{B_1^{d}(0)}\rho\,\de \by=1$ for all $\vep>0$,
    for every $x\in\rplusd$ and $\vep>0$,  
    we~find~that
\begin{equation*}
  \begin{aligned}
    |S_{\vep}(\bun)(\bx)-\bun(\bx)|&=\bigg| \int_{B_{\vep}^{d}(0)}\rho_{\vep}(\by)\,
    {\overline{\bun}}(\bx-2\vep\, \mathrm{e}_d-\by)\dy-\bun(\bx)\bigg|
    \\
    & =\bigg|\int_{B_{\vep}^{d}(0)}\rho_{\vep}(\by)( 
    {\overline{\bun}}(\bx-2\vep\, \mathrm{e}_d-\by)-{\overline{\bun}}(\bx))\,\dy\bigg|
    \\
    &\leq
    [\overline{\bun}]_{\dot{C}^{0,\alpha}(\R^d)}\int_{B_{\vep}^{d}(0)}\rho_{\vep}(\by)|\by
    +2\vep\,\mathrm{e}_d|^{\alpha}\,\de \by
    \\
    &\leq [\bun]_{{\dot C}^{0,\alpha}(\overline{\rplusd})}\, (3\vep)^{\alpha}\,.
    \end{aligned}
\end{equation*}
\textit{ad \eqref{eq:conv3p}.} Using that
$\int_{B_{\vep} ^{d}(0)} \nabla \rho_{\vep}\,\mathrm{d}y=\mathbf{0}$ for all $\vep>0$, for every
$x\in \rplusd$ and $\vep>0$, we find that
\begin{align*}
  \begin{aligned}
    \vert \nabla S_{\vep}(\bun)(x)\vert &=\vert (\nabla
    \rho_{\vep}\ast\tau_{2\vep}\overline{\bun})(x)\vert
    \\
    &=\frac{1}{\vep} \bigg\vert \int_{B_{\vep} ^{d}(x)}\nabla \rho_{\vep}(x-y)
    (\overline{\bun}(y- 2\vep\,\mathrm{e}_d)-\overline{\bun}(x- 2\vep\,\mathrm{e}_d))
    \dy\bigg\vert
    \\
    &\leq [\overline{\bun}]_{\dot{C}^{0,\alpha}(\R^d)}\,\vep ^{\alpha-1} \int_{B_{\vep}^{d}(x)}\vert (\nabla
    \rho)_{\vep}(x-y)\vert \dy
    \\
    &= [\bun]_{{\dot C}^{0,\alpha}(\overline{\rplusd})}\,\vep ^{\alpha-1} \int_{B_1^{d}(0)}\vert \nabla \rho(y)\vert \dy\,. 
  \end{aligned}
\end{align*}
\vspace*{-7.5mm}

\end{proof}
Since we will study a time-evolution problem, we report also a basic result needed to
justify the (space) smoothing in Bochner--Lebesgue spaces. The proposition below makes rigorous the
calculations we will employ later on, showing that --beside becoming smooth in the space
variables-- certain Bochner--Sobolev properties are preserved by both the standard
mollification and also by the convolution-translation~operator.\enlargethispage{5mm}
\begin{proposition}\label{prop:derivate_eps}
  Let $d,\ell\in \mathbb{N}$ and $s,r\in [1,\infty]$. Moreover, for every
  $\bun\in L^s(0,T;(L^r(\rplusd))^{\ell})$ and $\vep>0$, define
  $\bun_{\vep}\colon (0,T)\to (L^r(\R^d))^{\ell}\cap (C^\infty(\R^{d}))^{\ell}$ by
 	\begin{align*}
 		\bun_{\vep}(t)\coloneqq (\bun(t))_{\vep}\quad\text{ for a.e.\ }t\in (0,T)\,.
 	\end{align*}
 	Then, the following statements apply:
        \begin{itemize}[noitemsep,topsep=2pt,leftmargin=20pt,labelwidth=\widthof{(iii)}]
        \item[(i)] For every $\vep>0$, it holds that $\bun_{\vep}\in L^s(0,T;(L^r(\rplusd))^{\ell})$.
        \item[(ii)] If, in addition, $\bun\in W^{1,s}(0,T;(L^r(\rplusd))^{\ell})$, then, for every $\vep>0$, it holds that
          $\bun_{\vep}\in W^{1,s}(0,T;(L^r(\R^{d}))^{\ell})$ with
          \begin{align*}
            \frac{\mathrm{d}\bun_{\vep}}{\mathrm{d}t}(t)=
            \bigg(\frac{\mathrm{d}\bun}{\mathrm{d}t}(t)\bigg)_{\vep}  
            \quad\text{ in }(L^r(\R^{d}))^{\ell}\quad\text{ for a.e.\ }t\in (0,T)\,. 
 		\end{align*}
 	\end{itemize}
 \end{proposition}
 \begin{proof}
    Since $(\bun\mapsto \bun_\vep)\colon (L^r(\rplusd))^{\ell}\to (L^r(\R^{d}))^{\ell}$ is a linear and continuous operator, 
 claim (i) follows from~\cite[Proposition 1.2.2]{Droniou2023}
   and claim (ii) from~\cite[Proposition~2.5.1]{Droniou2023}.
 \end{proof}
\begin{proposition}
    Let $d,\ell\in \mathbb{N}$ and $s,r\in [1,\infty]$. Moreover, for every
    $\bun\in L^s(0,T;(L^r(\rplusd))^{\ell})$ and $\vep>0$, define
    $S_{\vep}(\bun)\colon (0,T)\to (L^r(\R^s))^{\ell}\cap (C^\infty(\R^{d}))^{\ell}$~by
    \begin{align*}
      S_{\vep}(\bun)(t)\coloneqq S_{\vep}(\bun(t))\quad\text{ for a.e.\ }t\in (0,T)\,.
 	\end{align*}
	Then, the following statements apply:
 	\begin{itemize}[noitemsep,topsep=2pt,leftmargin=22pt,labelwidth=\widthof{(iii)}]
        \item[(i)] For every $\vep>0$, it holds that $S_{\vep}(\bun)\in L^s(0,T;(L^r(\rplusd))^{\ell})$.
        \item[(ii)] If, in addition, $\ell=d$ and $\bun\in L^s(0,T;W^{1,r}_{0,\sigma}(\rplusd))$, then, for every
          $\vep>0$, it holds that $S_{\vep}(\bun)\in L^s(0,T;W^{1,r}_{0,\sigma}(\rplusd))$ with
          $\bun_{\vep}(t)\in C^\infty_{\sigma}(\rplusd)$ and $\bun_{\vep}(t)_{|\partial\rplusd}=\mathbf{0}$ for a.e.\ $t\in (0,T)$.
 	\end{itemize}
      \end{proposition} 
      \begin{proof}
        The proof follows as in the previous proposition, by using, in addition, that the
        translation $\tau_{2\vep}\colon (L^r(\R^d))^{\ell}\to (L^r(\R^{d}))^{\ell}$ is a linear and continuous operator. 
      \end{proof}
      \section{Proof of Theorem~\ref{thm:main}.}
      \label{sec:NSE}
      In this section, we prove the main result of this paper, namely
      Theorem~\ref{thm:main}. 
      
      To begin with, if we integrate the inner product of \eqref{eq:Navier-Stokes_strong}$_1$ with
       $\rho_{\vep}* S_{\vep}(\bvn)$ over $(0,t)\times \rplus$ for all
       $t\in (0,T)$ and $\vep>0$, use Green's first formula for
       Sobolev functions in space for a.e.\ $\tau\in (0,t)$, the
       divergence constraint \eqref{eq:Navier-Stokes_strong}$_2$, the self-adjointness and commutator properties  of the mollification operator,  
      for every $t\in [0,T]$ and $\vep >0$, we find that
\begin{align}
  \label{eq:main.1}
  \begin{aligned}
    \int_{0}^{t}{\int_{\rplus}{\bigg(\frac{\mathrm{d}\bvn}{\mathrm{d}t}\bigg)_{\vep}\cdot
         S_{\vep}(\bvn)\,\mathrm{d}x}\,\de\tau}&+\nu
    \int_{0}^{t}{\int_{\rplus}{\nabla \bvn_{\vep}:\nabla
         S_{\vep}(\bvn)\,\mathrm{d}x}\,\de\tau}
    \\
    &-\int_{0}^{t}{\int_{\rplus}{(\bvn\otimes
        \bvn)_{\vep} :\nabla
         S_{\vep}(\bvn)\,\mathrm{d}x}\,\de\tau}=0\,.
  \end{aligned}
\end{align}
The  objective is to establish that \eqref{eq:main.1} for all $t\in (0,T)$ converges (as $\vep\to0^{+}$)~to~\eqref{eq:energy-equality}. In
particular, the latter will follow establishing that the term containing the time-derivative
converges for all $t\in (0,T)$   to $\frac{1}{2}\|\bvn(t)\|_{L^2(\rplus)}^{2}-\frac{1}{2}\|\bvn(0)\|_{L^2(\rplus)}^{2}$ (as $\vep\to0^{+}$) and that the non-linear term converges to zero (as $\vep\to0^{+}$). Both statements cannot be justified
directly since the available regularity of the velocity vector field is not enough to directly
evaluate the integrals with $\bvn$ instead of $ S_{\vep}(\bvn)$, thus, obtaining the result
simply by using that convolution-translation is an approximation~of~the~identity. We will
show that they converge to the right objects, due to the particular assumptions of
regularity and also due to the specific type of smoothing we use.
\begin{remark}
  The regularity assumptions that will directly justify the convective term to vanish are
  the same as in the classical results by Shinbrot (\textit{cf}.~\cite{BY2019,BC2020,Gal2000a}). Similar behavior and regularity will be requested for
  the time-derivative. So the computations that follow are particularly delicate since we will
  prove the energy equality in a class of possibly non-smooth weak solutions.
\end{remark}
We split the proof into several parts.  We start with a preliminary result about the
behavior of the term involving the time-derivative.
\begin{proposition}
  \label{prop:timed}
  Let  $\bvn\in L^2(0,T;W^{1,2}_{0,\sigma}(\rplus))\cap L^\infty(0,T;L^2_{0,\sigma}(\Omega))$
  be a Leray--Hopf (weak) solution (in the sense of
  Definition \ref{def:leray_hopf}). Moreover, let
  \begin{align*}
      \bvn\in  L^\beta(0,T;(C^{0,\alpha}(\overline\rplus))^3)\,,
  \end{align*}
  where $\alpha \in (0,1)$ and
  $\beta \in (\beta_0, 2)$ with
  $\beta_0\coloneqq\frac{6}{3+2\alpha}\ge 1$, and let
  \begin{equation*}
  \bvn_{0}\in
  X_{s,r}\text{  for some
}  r\in(\frac{4}{2+\beta},\frac{4}{4-\beta})\ 
 \text{and with } s:=\frac{2r\beta}{4r+r \beta-4}\,.
\end{equation*}
Then, it holds that
  \begin{equation*}
\int_0^T\int_{\rplus}\bigg(\frac{\mathrm{d}\bvn}{\mathrm{d}t}\bigg)_{\vep}
  \cdot\big(\bvn_{\vep}- S_{\vep}(\bvn)\big)\dx\, \de\tau\to 0\quad (\vep \to0^+)\,. 
\end{equation*}
\end{proposition}
\begin{remark}
  Observe that in the case without boundaries, one can smooth the velocity by a double convolution and
  so the smoothing operator becomes self-adjoint. In this case, one can show directly
  (with the aid of a further independent convolution regarding only the time variable) that
  \begin{align*}
    \left.\begin{aligned} \int_{0}^{t}\int_{\T^3}\bvn\cdot\partial_{t}(\rho_{\vep}*\bvn_{\vep})\dx\, \de\tau&=
    \frac{1}{2}\|\bvn_{\vep}(t) \|^{2}_{L^{2}(\T^3)}-
    \frac{1}{2}\|\bvn_{\vep}(0)\|^{2}_{L^{2}(\T^3)}\\&\to     \frac{1}{2}\|\bvn(t) \|^{2}_{L^{2}(\T^3)}-
    \frac{1}{2}\|\bvn_{0}\|^{2}_{L^{2}(\T^3)}\end{aligned}\quad\right\}
    \quad( \vep\to 0^+)\,.
  \end{align*}
  In our case, we need to adapt this procedure and additional regularity of the
  time-derivative is requested.
\end{remark}
\begin{proof}[Proof (of Proposition~\ref{prop:timed})]
  To begin with, let $q\in (\frac{4}{2+\beta},2)$ be fixed, but arbitrary. Then, we define the exponent $s\coloneqq s(q)\in (1,+\infty)$ via
   \begin{align}\label{def:s}
  	s\coloneqq\frac{q \beta }{-2+q+q \beta }\,,\qquad\textit{i.e.},\; s'=\frac{\beta q}{2-q}\,,
  \end{align} 
  so that, due to $\bvn\in L^{\beta}(0,T;C^{0,\alpha}(\overline{\rplus}))$, we have that
  $(t\mapsto [\bvn(t)]_{\dot{C}^{0,\alpha}(\overline{\rplus})}^{\frac{2}{q}-1})\in L^{s'}(0,T)$.
  Note that the restriction $q\in (\frac{4}{2+\beta},2)$ ensures
  that $s \in (1,2)$.
  Then, by Corollary~\ref{cor:alternative_formulation}, we have that $\frac{\mathrm{d}\bvn}{\mathrm{d}t}\in L^s(0,T;(L^r(\rplus))^3)$,
  where 
  \begin{align}\label{def:r}
  	r=\frac{4s}{4s+\beta s-2\beta}=\frac{4q}{4+2q-q \beta}\,,
  \end{align}
  where we used \eqref{def:s}. Note that $q\in
  (\frac{4}{2+\beta},2) $  and $\beta <2$ guarantee that $r\in (1,q)$.
  
  Hence, 
  using H\"older's inequality, standard convex interpolation, adding
  and subtracting $\bvn(\tau)$ in the second term in the second line,
  Lemma~\ref{lem:convolution-Holder0}\eqref{eq:conv20},
  Lemma~\ref{lem:convolution-Holder}\eqref{eq:conv2p}, and standard
  properties of the convolution operator (cf.~\cite[Sec.~A.3]{RRS16}) and the translation operator,
  as well as the
  energy inequality~\eqref{def:leray_hopf.2}, for a.e.\ $\tau \in (0,T)$ and every $\vep>0$, we
  obtain 
\begin{align}\label{prop:timed.1}
 \begin{aligned}
    &\int_{\rplus}{\bigg\vert\bigg(\frac{\mathrm{d}\bvn}{\mathrm{d}t}(\tau)\bigg)_{\vep}\bigg\vert\vert
     S_{\vep}(\bvn)(\tau)-\bvn_{\vep}(\tau) \vert\,\mathrm{d}x}
    \\&\leq\bigg\|\bigg(
     \frac{\mathrm{d}\bvn}{\mathrm{d}t}(\tau)\bigg)_{\vep}\bigg\|_{L^q(\rplus)}   
    \| S_{\vep}(\bvn)(\tau)-\bvn_{\vep}(\tau)
      \|_{L^\infty(\rplus)}^{\frac{2}{q}-1}
\|S_{\vep}(\bvn)(\tau)-\bvn_{\vep}(\tau)
      \|_{L^2(\rplus)}^{\frac{2}{q'}}
    \\
    &\leq c\,\vep ^{-3(\frac{1}{r}-\frac{1}{q})}\,
      \bigg\|\frac{\mathrm{d}\bvn}{\mathrm{d}t}(\tau)\bigg\|_{L^{r}(\rplus)}
      \smash{[\bvn(\tau)]_{\dot{C}^{0,\alpha}(\overline{\rplus})}^{\frac{2}{q}-1}}
     \vep^{\alpha(\frac{2}{q}-1)}\|\bvn_0\|_{L^2(\rplus)}^{\frac{2}{q'}}\,.
      \end{aligned}
  \end{align}
  Then, integrating \eqref{prop:timed.1} with respect to $\tau\in (0,T)$ and applying 
  H\"older's inequality, for every $\vep>0$, we find that 
  \begin{align}\label{prop:timed.2}
   \begin{aligned}
   	&\int_{0}^{T}  {\int_{\rplus}{\bigg\vert\bigg(
   			\frac{\mathrm{d}\bvn}{\mathrm{d}t}\bigg)_{\vep}\bigg\vert\vert
   			S_{\vep}(\bvn)-\bvn_{\vep} \vert\,\mathrm{d}x}\, \de\tau}   
   	\\&\leq c \,\frac{\vep^{\alpha(\frac{2}{q}-1)}}{\vep^{3(\frac{1}{r}-\frac{1}{q})}}
   	\bigg\|\frac{\mathrm{d}\bvn}{\mathrm{d}t}\bigg\|_{L^s(0,T;L^r(\rplus))}
   	\|\bvn\|_{L^{\beta}(0,T;C^{0,\alpha}(\overline{\rplus}))}^{\frac{2}{q}-1}\,.
   \end{aligned}
  \end{align}  
In order to guarantee that the right-hand side in \eqref{prop:timed.2} vanishes if we pass to the limit~for~$\vep\to 0^{+}$,
    it is sufficient that
    \begin{align}\label{eq:est-eps}
    \alpha\left(\frac{2}{q}-1\right)> 3\left(\frac{1}{r}-\frac{1}{q}\right)=3\left(\frac{2-\beta}{4}\right)\,,
    \end{align}
    where we used \eqref{def:r}. 
	Since $(q\mapsto \frac{2}{q})\colon (1,2) \to (1,2)$ is decreasing, the less restrictive case (larger
        $\beta$) is the case when $q$ is close to the infimum $\frac{4}{2+\beta}$. Since, in this case, we have that
    \begin{align*}
        \alpha\bigg(\frac{2}{q}-1\bigg)_{\smash{|q=\frac{4}{2+\beta}}}=\frac{\alpha 
          \beta}{2}\,,
    \end{align*}
     then the inequality~\eqref{eq:est-eps} turns out to be satisfied for
	\begin{equation*}
          \beta>\frac{6}{3+2\alpha}.
        \end{equation*}
	The choice of $r,q$ such that $1<\frac{4}{2+\beta}<r<q<\frac{4}{4-\beta}$
	%
	%
        is always possible for $\beta \in (1,2)$, since
	$\frac{4}{2+\beta}<\frac{4}{4-\beta}$.\enlargethispage{7.5mm}
 
    Eventually, due to inequality~\eqref{eq:est-eps}, from \eqref{prop:timed.2}, it follows that
        \begin{equation*}
           \int_{0}^{T}  {\int_{\rplus}{\bigg\vert
              \bigg(
      \frac{\mathrm{d}\bvn}{\mathrm{d}t}\bigg)_{\vep}\bigg\vert\,\vert
               S_{\vep}(\bvn)-\bvn_{\vep} \vert\,\mathrm{d}x}\, \de\tau}\to0\quad(\vep\to0^{+})\,,
	\end{equation*}
 which concludes the proof.
\end{proof}
 
We can now prove the main result of the paper.
\begin{proof}[Proof (of Theorem~\ref{thm:main}).] We study the behavior as $\vep\to0^{+}$
  of the terms in~\eqref{eq:main.1}:
  
  \medskip
  \noindent\textit{1.\ Convergence of family of integrals involving the
    time derivative:}~\mbox{Proposition~\ref{prop:derivate_eps}} yields that $\bvn_{\vep}\in W^{1,s}(I;(L^r(\R^3))^3)$ for all $\vep>0$. Since $\bvn\in L^\infty(I;L^2_{0,\sigma}(\rplus))$, we have that
    $\bvn_\vep\in L^\infty(I;(W^{2,2}(\R^3))^3)$ for all $\vep>0$, which, by a Sobolev embedding  (\textit{cf}.\ \cite[Thm.\ 4.12, Part I, Case A]{AdamsFournier03}), implies that $\bvn_\vep\in L^\infty(I;(L^\infty(\R^3))^3)$ for all $\vep>0$. Therefore, by real interpolation, we find that $\bvn_\vep\in L^\infty(I;(L^p(\R^3))^3)$ for all $p\in [2,\infty]$ and for all $\vep>0$. 
    In summary, we arrive at  $\bvn_\vep\in W^{1,s',s}_{\textup{id}_{\smash{L^{\smash{r'}}}}}(I; (L^{r'}(\rplus))^3, (L^r(\rplus))^3)$ for all $\vep>0$,~so that, resorting to Proposition \ref{prop:non_sym_pi}(ii) (in the case $X=(L^{r'}(\rplus))^3$, $Y=(L^2(\rplus))^3$, $i=\textup{id}_{\smash{L^{\smash{r'}}}}$, $p=s'$, and $q=s$), for every $t\in [0,T]$, we find that
    \begin{align}\label{eq:main.0.-1}
    \int_0^{t}{\int_{\rplus}{\frac{\mathrm{d}\bvn_{\vep}}{\mathrm{d}t}\cdot
    		\bvn_{\vep}\,\mathrm{d}x}\,\de\tau}=\frac{1}{2}\|\bvn_{\vep}(t)\|_{L^2(\rplus)}^2-\frac{1}{2}\|\bvn_{\vep}(0)\|_{L^2(\rplus)}^2\,.
    \end{align}
    
    Then, for every~${t\hspace{-0.1em}\in \hspace{-0.1em}[0,T]}$, using Proposition \ref{prop:derivate_eps}(ii) and \eqref{eq:main.0.-1}, we infer that
  \begin{align}\label{eq:main.0.0}
    \begin{aligned}
      &\int_0^{t}{\int_{\rplus}{\bigg(
      \frac{\mathrm{d}\bvn}{\mathrm{d}t}\bigg)_{\vep}\cdot
           S_{\vep}(\bvn)\,\mathrm{d}x}\,\de\tau}
      \\
      &=
      \int_0^{t}{\int_{\rplus}{\frac{\mathrm{d}\bvn_{\vep}}{\mathrm{d}t}\cdot
          \bvn_{\vep}\,\mathrm{d}x}\,\de\tau}
      +
      \int_0^{t}{\int_{\rplus}{\bigg(
      \frac{\mathrm{d}\bvn}{\mathrm{d}t}\bigg)_{\vep}\cdot
          ( S_{\vep}(\bvn)-\bvn_{\vep})\,\mathrm{d}x}\,\de\tau}
      \\
      &=\bigg[\frac{1}{2}\|\bvn_{\vep}(t)\|_{L^2(\rplus)}^2-\frac{1}{2}\|\bvn_{\vep}(0)\|_{L^2(\rplus)}^2\bigg]
      \\
      &\quad
      +\int_0^{t}{\int_{\rplus}{\bigg(
      \frac{\mathrm{d}\bvn}{\mathrm{d}t}\bigg)_{\vep}\cdot
          ( S_{\vep}(\bvn)-\bvn_{\vep})\,\mathrm{d}x}\,\de\tau}
      \\
      &\eqqcolon I_{\vep} ^1(t)+I_{\vep} ^2(t)\,.
    \end{aligned}
  \end{align}
  Thus, we need to estimate $I_{\vep} ^1(t)$ and $I_{\vep} ^2(t)$: %

\textit{ad $i=1$.}
  Using that $\bvn\in C^0_w([0,T];L^2_{0,\sigma}(\rplus))$ (\textit{cf}.\ Definition \ref{def:leray_hopf}) and  approximation properties of standard mollification, for every $t\in [0,T]$, we find that
  \begin{align}\label{eq:main.0.1}
I_{\vep} ^1(t)\to
    \frac{1}{2}\|\bvn(t)\|_{L^2(\rplus)}^2-\frac{1}{2}\|\bvn(0)\|_{L^2(\rplus)}^2
    \quad (\vep \to 0^{+})
    \,.
	\end{align}

	\textit{ad $i=2$.}
        By Proposition~\ref{prop:timed}, for every $t\in [0,T]$,
        we have that
        \begin{align}\label{eq:main.0.2}
            |I_{\vep}^{2}(t)|\leq |I_{\vep}^{2}(T)|\to0\quad (\vep \to 0^{+})\,.
        \end{align} 
    As a result, using \eqref{eq:main.0.1} and \eqref{eq:main.0.2} in \eqref{eq:main.0.0}, for every $t\in [0,T]$, we arrive at
    \begin{align}\label{eq:main.0.3}
    \begin{aligned}
      \int_0^{t}{\int_{\rplus}{\bigg(\frac{\mathrm{d}\bvn}{\mathrm{d}t}\bigg)_{\vep}\cdot
           S_{\vep}(\bvn)\,\mathrm{d}x}\,\de\tau}\to  \frac{1}{2}\|\bvn(t)\|_{L^2(\rplus)}^2-\frac{1}{2}\|\bvn(0)\|_{L^2(\rplus)}^2\quad (\vep \to 0^{+})\,.
    \end{aligned}
  \end{align}
    
        \begin{remark}
          Note that the condition $\beta>\frac{6}{3+2\alpha}$ is required to estimate this
          term, which  does not involve the non-linear convective term. The convective term
          requires the same weaker assumptions as in the space periodic case, that is
          $\beta\geq\frac{2}{1+\alpha}$,~\textit{cf}.~\eqref{eq:holder-periodic}.
        \end{remark}
        \medskip
        
        \noindent\textit{2. Convergence of family of integrals arising from
          the viscous term:} Recalling that $\bvn\in L^2(0,T;W^{1,2}_{0,\sigma}(\rplus))$ (\textit{cf}.\ Definition \ref{def:leray_hopf}), appealing to
        Proposition~\ref{prop:basic_properties}(iii)~\&~(iv), it holds that
        $\| S_{\vep}(\bvn)(t) \|_{\smash{W^{1,2}_{0,\sigma}(\rplus)}}\leq c_2
        \,\|\bvn(t)\|_{\smash{W^{1,2}_{0,\sigma}(\rplus)}}$~and 
	$  S_{\vep}(\bvn)(t) \to \bvn (t)$ in
        $W^{1,2}_{0,\sigma}(\rplus)$ $(\vep \to 0^{+})$ for a.e.\
        $t\in (0,T)$. Hence,~by~Lebesgue's~dominated convergence
        theorem for Bochner--Lebesgue functions, we obtain\enlargethispage{5mm}
        \begin{align}
          \label{eq:main.2}
           S_{\vep}(\bvn)\to \bvn \quad \text{ in }L^2(0,T;W^{1,2}_{0,\sigma}(\rplus))\quad(\vep  \to 0^{+})\,. 
	\end{align}
        On the other hand, by standard approximation properties of convolution,~it~holds~that
        $\|\bvn_{\vep}(t) \|_{\smash{W^{1,2}(\R^3)}}\leq \|\overline{\bvn(t)}\|_{\smash{W^{1,2}(\R^3)}}$ and
        $ \bvn_{\vep}(t) \to \overline{\bvn (t)}$ in $(W^{1,2}_0(\R^3))^3$ $(\vep \to 0^{+})$~for~a.e.\ $t\in (0,T)$. Hence, by Lebesgue's dominated convergence
        theorem for Bochner--Lebesgue functions, we obtain
	\begin{align}
          \label{eq:main.3}
		\bvn_{\vep}\to \overline{\bvn} \quad \text{ in
          }L^2(0,T;(W^{1,2}_0(\mathbbm{R}^3))^3)\quad(\vep  \to 0^{+})\,.
	\end{align}
	As a result, using~\eqref{eq:main.2} and~\eqref{eq:main.3},
        for every $t\in [0, T]$, we conclude that
        \begin{align}
          \label{eq:main.4}
          \nu 	\int_{0}^{t}{\int_{\rplus}{\nabla \bvn_{\vep}:\nabla
           S_{\vep}(\bvn)\,\mathrm{d}x}\,\de\tau}\to \nu
          \int_{0}^{t}{\|\nabla \bvn\|_{L^2(\rplus)}^2\,\de\tau}\quad(\vep
          \to 0^{+})\,. 
	\end{align} 
	 
    \begin{remark}
        Note that for the convergence \eqref{eq:main.4}, we used only the regularity known for all Leray--Hopf (weak)
        solutions (in the sense of Definition \ref{def:leray_hopf}).
    \end{remark}

        \medskip
        \noindent \textit{3. Convergence of  family of integrals arising from
          the convective term:} Recalling that $\bvn\in L^2(0,T;W^{1,2}_{0,\sigma}(\rplus))\cap L^\infty(0,T;L^2_{0,\sigma}(\rplus))$ (\textit{cf}.\ Definition \ref{def:leray_hopf}), appealing to
        Proposition~\ref{prop:basic_properties}(iii) \& (iv), for
        a.e.\ $\tau\in (0,T)$, it holds that
	\begin{align}
          \label{eq:main.5}
           S_{\vep}(\bvn)(\tau) \to \bvn (\tau)\quad \text{ in }W^{1,2}_{0,\sigma}(\rplus)\quad
          (\vep  \to 0^{+})\,.
	\end{align}
        On the other hand, resorting to approximation properties of
        standard mollification, for a.e.\ $\tau\in (0,T)$ and every $q\in[1,3]$, it holds that
        \begin{align}
          \label{eq:main.6}
		(\bvn(\tau)\otimes\bvn(\tau))_{\vep}\to
          \overline{\bvn(\tau)\otimes\bvn(\tau)} \quad \text{ in
          }
          (L^{q}(\mathbbm{R}^3))^{3\times 3}\quad(\vep  \to 0^{+})\,.
        \end{align}
	As a result, using~\eqref{eq:main.5} and~\eqref{eq:main.6}, for a.e.\ $\tau\in (0,
        T)$, we conclude that
        \begin{align}\label{eq:main.7}
		\left.\begin{aligned}
                & \int_{\rplus}{(\bvn(\tau)\otimes \bvn(\tau))_{\vep}
                      :\nabla S_{\vep}(\bvn)(\tau) \,\mathrm{d}x}\\&\quad \to
                    \int_{\rplus}{\bvn(\tau)\otimes \bvn(\tau):\nabla\bvn
                      (\tau)\,\mathrm{d}x}
                         =0 
		\end{aligned}\quad\right\}\quad(\vep  \to 0^{+})\,.
	\end{align}
	To be in the position to apply Lebesgue's convergence
        theorem, it is left to~find~an $L^1(0,T)$-integrable majorant for the
        left-hand side in~\eqref{eq:main.7}.  To this end, we resort~to the Constantin--E--Titi commutator-type identity (\textit{cf}.\ \cite{CET1994}), which
        states that~for~a.e.~${\tau\in (0,T)}$ and every $\vep>0$, it holds that
        \begin{align}
         \label{eq:main.8}
		\begin{aligned}
                  (\bvn(\tau)\otimes \bvn(\tau))_{\vep} &=\bvn_{\vep}(\tau)
                  \otimes \bvn_{\vep}(\tau) \\&\quad+r_{\vep} (\bvn(\tau),\bvn(\tau))
                  \\
                  &
                  \quad-\big(\bvn(\tau)-\bvn_{\vep}(\tau)\big) \otimes
                  \big(\bvn(\tau)-\bvn_{\vep}(\tau)\big) \,,
		\end{aligned}
	\end{align}
	where for a.e.\ $(\tau,x)^\top\in (0,T)\times \mathbbm{R}^d$ and every $\vep>0$, we employed the abbreviations
	%
\begin{align*}
          \left\{
          \begin{aligned}
                  r_{\vep} (\mathbf{v},\mathbf{v})(\tau,x)&\coloneqq
                  \int_{\mathbbm{R}^3}{\rho_{\vep
                    }(y)(\delta_y\mathbf{v}
                    )(\tau,x)\otimes(\delta_y\mathbf{v})
                    (\tau,x)\,\mathrm{d}y}\,,
                  \\
                  \delta_y\mathbf{v} (\tau,x)&\coloneqq
                  \overline{\mathbf{v}}(\tau,x-y)-\overline{\mathbf{v}}(\tau,x)\quad\text{ for a.e.\
                  }y\in \rplus \,,
		\end{aligned}
                                          \right.
        \end{align*}
	Using the Constantin--E--Titi commutator-type identity \eqref{eq:main.8}, the left-hand~side~in~\eqref{eq:main.7}, for a.e.\ $\tau\in (0,T)$ and every $\vep>0$, can be re-written as
        \begin{align}
          \label{eq:main.10}
		\hspace*{-5mm}\begin{aligned}
                  &\int_{\rplus}{(\bvn(\tau)\otimes \bvn(\tau))_{\vep}
                    :\nabla S_{\vep}(\bvn)(\tau)\,\mathrm{d}x}
                  \\
                  &=\int_{\rplus}{\bvn_{\vep}(\tau)
                    \otimes \bvn_{\vep}(\tau):\nabla S_{\vep}(\bvn)(\tau) \,\mathrm{d}x}
                  \\&\quad+\int_{\rplus}{r_{\vep}
                    (\bvn(\tau),\bvn(\tau)):\nabla S_{\vep}(\bvn)(\tau)\,\mathrm{d}x}
                  \\
                  &\quad-\int_{\rplus}\Big\{\big(\bvn(\tau)-\bvn_{\vep}(\tau)\big)\otimes
                    \big(\bvn(\tau)-\bvn_{\vep}(\tau) \big):\nabla S_{\vep}(\bvn)(\tau)\Big\}\,\mathrm{d}x
                  \\
                  &\eqqcolon J^{1}_{\vep }(\tau)+J_{\vep}^{2}(\tau)+J_{\vep}^{3}(\tau)\,.
		\end{aligned}
	\end{align}
        %
	So, we need to establish that $| J^{i}_{\vep }(\tau)|$, $i=1,2,3$, are for a.e.\
        $\tau\in (0,T)$ is bounded from above by an $L^1(0,T)$-integrable function:

        \textit{ad  $ i=1$.} Using that
        $\int_{\rplus}{S_{\vep}(\bvn)(\tau) \otimes\bvn_{\vep}(\tau) :\nabla S_{\vep}(\bvn)(\tau)
          \,\mathrm{d}x}=0$ for a.e.\ $\tau\in (0,T)$, due to the
        properties of the convolution-translation operator
        (cf.~Proposition \ref{prop:basic_properties}), adding
        and subtracting $\bvn$,  the properties of the
        convolution operator and the translation-convolution operator
        (cf.~Lemma~\ref{lem:convolution-Holder0}\eqref{eq:conv20},
  Lemma~\ref{lem:convolution-Holder}\eqref{eq:conv2p}, \eqref{eq:conv3p}),  H\"older inequality with the
        exponents $(2,\frac{2}{\alpha},\frac{2}{1-\alpha})$, and again
        the properties of the
        convolution operator and the translation-convolution operator
        (cf.~\cite[Sec.~A.3]{RRS16}, 
  Proposition~\ref{prop:basic_properties}\eqref{prop:basic_properties.3},
  \eqref{prop:basic_properties.4}) together with \eqref{def:leray_hopf.2}, for a.e.\ $\tau\in (0,T)$, we find that
	\begin{align}\label{eq:main.11}
		\begin{aligned}
                  \hspace*{-6mm}
                  \vert J^{1}_{\vep
                  }(\tau)\vert&
                  =\bigg\vert-\int_{\rplus}{( S_{\vep}(\bvn)(\tau)-\bvn_{\vep}(\tau)
                    ) \otimes\bvn_{\vep}(\tau) :\nabla S_{\vep}(\bvn)(\tau) \,\mathrm{d}x}\bigg\vert
                  \\
                  &\leq \int_{\rplus}\Big\{\vert \bvn_{\vep}(\tau) \vert
                  \,\vert  S_{\vep}(\bvn)(\tau)-\bvn_{\vep}(\tau)
                  \vert^{1-\alpha}\vert  S_{\vep}(\bvn)(\tau)-\bvn_{\vep}(\tau) \vert^{\alpha}
                  \\
                  &\qquad\quad\times\vert \nabla S_{\vep}(\bvn)(\tau)
                  \vert^{1-\alpha}\vert \nabla S_{\vep}(\bvn)(\tau)
                  \vert^{\alpha}\Big\} \,\mathrm{d}x
                  \\
                  &\leq
                  c\,[\bvn(\tau)]_{\dot{C}^{0,\alpha}(\overline{\rplus})} ^{\alpha+1-\alpha}\vep^{\alpha(1-\alpha)+\alpha(\alpha-1)}
                  \\
                  &\quad\times\int_{\rplus}{\vert \bvn_{\vep}(\tau)
                    \vert\,\vert  S_{\vep}(\bvn)(\tau) -\bvn_{\vep}(\tau)
                    \vert^{\alpha} \vert \nabla S_{\vep}(\bvn)(\tau)
                    \vert^{1-\alpha} \,\mathrm{d}x}
                  \\
                  &\leq c\,[\bvn(\tau)]_{\dot{C}^{0,\alpha}(\overline{\rplus})} 
                  \\
                  &\quad\times
                  \,\|\bvn_{\vep}(\tau) \|_{L^2(\rplus)}
                  \| S_{\vep}(\bvn)(\tau) -\bvn_{\vep}(\tau)
                  \|_{L^2(\rplus)}^{\alpha} \|\nabla S_{\vep}(\bvn)(\tau)
                  \|_{L^2(\rplus)}^{1-\alpha}
                  \\
                  &\leq
                  c\,[\bvn(\tau)]_{\dot{C}^{0,\alpha}(\overline{\rplus})} \|\bvn_0\|_{L^2(\rplus)}^{1+\alpha}
                  \|\nabla\bvn(\tau)\|_{L^2(\rplus)}^{1-\alpha} \,.
		\end{aligned}\hspace*{-15mm}
	\end{align}

 
	\textit{ad $i=2$.} Using Fubini-Tonelli theorems to exchange the order of integrals, the definition of H\"older continuity and the property~\eqref{eq:conv3p} of the operator of convolution-translation, for a.e.\
        $\tau\in (0,T)$, we obtain
        \begin{align*}
            \begin{aligned}
            \vert J_{\vep}^{2} (\tau)\vert &\leq
       		\int_{\rplus}\int_{B_{\vep}^d(0)}\rho_{\vep}(y){\vert
       			\delta_y(\bvn(\tau))\vert^2\, \vert\nabla
       			S_{\vep}(\bvn)(\tau) \vert \,\mathrm{d}y\,\mathrm{d}x}
          \\
          &\leq \int_{B_{\vep}^d(0)}\int_{\rplus}\Big\{\rho_{\vep}(y){\vert
       			\delta_y(\bvn(\tau))\vert^{1+\alpha}\vert
       			\delta_y(\bvn(\tau))\vert^{1-\alpha}}\\&\qquad\qquad\qquad\times\vert\nabla  S_{\vep}(\bvn)(\tau)\vert^{\alpha}
       			\vert\nabla  S_{\vep}(\bvn)(\tau)\vert^{1-\alpha}\Big\}\,\de x\,\mathrm{d}y
          \\
          &\leq c_{\rho}^{\alpha}[\bvn(\tau)]_{0,\alpha}^{\alpha}\vep^{\alpha(\alpha-1)}\\&\quad\times\int_{B_{\vep}^d(0)}\rho_{\vep}(y)|y|^{\alpha(1-\alpha)}{\int_{\rplus}\vert
       			\delta_y(\bvn(\tau))\vert^{1+\alpha}
       			\vert\nabla  S_{\vep}(\bvn)(\tau)\vert^{1-\alpha}\,\de x\,\mathrm{d}y},
            \end{aligned}
        \end{align*}
        Next, observing that $|y|<\vep$, 
        using H\"older inequality with respect to
        $\dx$, and the  translation invariance of the Lebesgue measure for arbitrary $y\in\R^3$ we get
        \begin{align}
        \label{eq:main.12}
            \begin{aligned}
                \vert J_{\vep}^{2} (\tau)\vert
                &
                \leq c_{\rho}^{\alpha} 2^{\alpha}[\bvn(\tau)]_{0,\alpha}^{\alpha}
                \|\nabla  S_{\vep}(\bvn)(\tau)\|_{L^2(\rplus)}^{1-\alpha}
                \\
                &\quad\times\int_{B_{\vep}^d(0)}\rho_{\vep}(y)\big(\|\bvn(\tau,\cdot-y)\|_{L^2(\rplus)}^{1+\alpha}+\|\bvn(\tau)\|_{L^2(\rplus)}^{1+\alpha}\big)
       			    \,\mathrm{d}y,
              \\
              &\leq c_{\rho}^{\alpha} 2^{\alpha+1}[\bvn(\tau)]_{0,\alpha}^{\alpha}
                \|\nabla  S_{\vep}(\bvn)(\tau)\|_{L^2(\rplus)}^{1-\alpha}\\&\quad\times \|\bvn(\tau)\|_{L^2(\rplus)}^{1+\alpha}\int_{B_{\vep}^d(0)}\rho_{\vep}(y)
       			    \,\mathrm{d}y
              \\
              &\leq c_{\rho}^{\alpha} 2^{\alpha+1}[\bvn(\tau)]_{0,\alpha}^{\alpha}
                \|\nabla \bvn(\tau)\|_{L^2(\rplus)}^{1-\alpha} \|\bvn_0\|_{L^2(\rplus)}^{1+\alpha},
            \end{aligned}
        \end{align}
        where in the last step we used also~\eqref{prop:basic_properties.4} and the energy inequality.

	\textit{ad $i=3$.}  Using the properties of the
        convolution operator and the translation-convolution operator
        (cf.~Lemma~\ref{lem:convolution-Holder0}\eqref{eq:conv20},
  Lemma~\ref{lem:convolution-Holder}\eqref{eq:conv3p}),  H\"older inequality with the
        exponents $(\frac{2}{1+\alpha},\frac{2}{1-\alpha})$, and again
        the properties of the
        convolution operator and the translation-convolution operator
        (cf.~\cite[Sec.~A.3]{RRS16}, 
  Proposition~\ref{prop:basic_properties}\eqref{prop:basic_properties.4}) together with
  \eqref{def:leray_hopf.2}, for a.e.\ $\tau\in (0,T)$, we find that
	%
        \begin{align}
          \label{eq:main.13}
          \begin{aligned}
            \vert J_{\vep}^{3}(\tau)\vert&\leq \int_{\rplus}{\Big\{ \vert
              \bvn(\tau)-\bvn_{\vep}(\tau) \vert^{1+\alpha} \vert
              \bvn(\tau)-\bvn_{\vep}(\tau) \vert^{1-\alpha}}
              \\&\qquad\quad\times {\vert
              \nabla S_{\vep}(\bvn)(\tau)\vert^{\alpha} \vert
              \nabla S_{\vep}(\bvn)(\tau)\vert^{1-\alpha}\Big\}\,\mathrm{d}x}
            \\
            &\leq [\bvn(\tau)]_{\dot{C}^{0,\alpha}(\overline{\rplus})} \int_{\rplus}{ \vert \bvn(\tau)
              -\bvn_{\vep}(\tau) \vert^{1+\alpha} \vert
              \nabla S_{\vep}(\bvn)(\tau)\vert^{1-\alpha} \,\mathrm{d}x}
            \\&\leq
            c_{\rho}^{\alpha}\,[\bvn(\tau)]_{\dot{C}^{0,\alpha}(\overline{\rplus})} \|\bvn_0\|_{L^2(\rplus)}^{1+\alpha}
            \|\nabla\bvn(\tau)\|_{L^2(\rplus)}^{1-\alpha} \,.
		\end{aligned}
	\end{align}

	Putting everything together, using Lebesgue's dominated convergence theorem in
        conjunction with~\eqref{eq:main.7} and~\eqref{eq:main.10} together
        with~\eqref{eq:main.11}--\eqref{eq:main.13}, for every $t\in [0, T]$, we conclude
        that
	\begin{align}\label{eq:main.14}
         \int_{0}^{t}{\int_{\rplus}{(\bvn\otimes
          \bvn)_{\vep} :\nabla S_\vep(\bvn)\,\mathrm{d}x}\,\de\tau}\to 0\quad (\vep\to0^{+})\,. 
	\end{align}
    Finally, combining \eqref{eq:main.0.3}, \eqref{eq:main.4}, and \eqref{eq:main.14}, from \eqref{eq:main.1}, for every $t\in [0,T]$,~it~follows~that
    \begin{align*}
          \frac{1}{2}\|\bvn(t)\|_{L^2(\rplus)}^2-\frac{1}{2}\|\bvn(0)\|_{L^2(\rplus)}^2+\nu\int_0^{t}\|\nabla\bvn(\tau)\|^2_{L^2(\rplus)}\, \de \tau=0\,,
    \end{align*}
    which is the claimed energy conservation.
\end{proof}
%
\bigskip

\appendix
\section{Basic properties of the smoothing by convolution-translation}
\label{sec:appendix}
In this section, we prove
Propositions~\ref{prop:basic_properties}--\ref{prop:admissibility},   the main
results on the modified smoothing operator obtained with the convolution-translation~from~Definition~\ref{def:uepsilonprime}.
\begin{proof}[Proof (of Proposition~\ref{prop:basic_properties}).]
  \textit{ad (i).} The smoothness $S_{\vep}(\bun) \in (C^\infty(\R^d))^{\ell}$~for~all~$\vep >0$
  is a consequence of the mollification. In addition, for every $\vep >0$, we have that
  \begin{align*}
    \textup{supp}\,S_{\vep}(\bun)&=\textup{supp}\,\tau_{2\vep}\overline{\bun}+\overline{B_{\vep}^{d}(0)} 
    \\
                               &=\textup{supp}\,\bun+2\vep\,\mathrm{e}_d+\overline{B_{\vep} ^{d}(0)}
    \\
                               &=\rplusd +2\vep\,\mathrm{e}_d+\overline{B_{\vep} ^{d}(0)}
    \\
                               &\subseteq \rplusd +\vep  \mathrm{e}_d\subseteq \rplusd\,.
  \end{align*}
  
  \textit{ad (ii).} By the chain rule and the commutator properties of the
  convolution operator, for every $\vep >0$, it holds that
  \begin{align}
    \nabla S_{\vep}(\bun) =	\rho_{\vep}\ast\tau_{2\vep}\overline{\nabla\bun}\quad \text{ in
    }\R^d\,,\label{eq:basic_properties.1} 
  \end{align}
  so that $\mathrm{div}\,S_{\vep}(\bun)=\rho_{\vep}\ast\tau_{2\vep}\overline{\mathrm{div}\,\bun}=0$ in   $\R^d$.
  
  \textit{ad (iii).}  
  
  \textit{ad \eqref{prop:basic_properties.1}.} For every $\vep >0$ and a.e.\
  $x\in \rplusd$, we obtain
  \begin{align}
    \label{eq:basic_properties.3}
    \begin{aligned}
      \vert (\rho_{\vep}\ast\tau_{2\vep}
      \overline{\bun})(x)\vert&\leq\int_{B_{\vep}^{d}(x)}{\rho_{\vep}(y-x+2\vep\,\mathrm{e}_d)
        \vert\overline{\bun}(y)\vert\,\mathrm{d}y}
      \\
      &\leq
      \|\rho\|_{L^\infty(\rplus)}\vert B_1^d(0)\vert\fint_{B_{\vep}^{d}(x)}{\vert\overline{\bun}(y)\vert\,\mathrm{d}y}
      \\
      &\leq c_d\,M(\vert\overline{\bun}\vert)(x)\,.
    \end{aligned}
  \end{align}
  
  \textit{ad \eqref{prop:basic_properties.2}.} Using~\eqref{eq:basic_properties.1}, similar
  to~\eqref{eq:basic_properties.3}, for every $\vep >0$ and a.e.\
  $x\in \rplusd $, we conclude that~\eqref{prop:basic_properties.2}
  holds true.

\textit{ad \eqref{prop:basic_properties.2.1}.}
  For every $\vep >0$ and a.e.\
  $x\in \rplusd $, we have that
		\begin{align}\label{eq:basic_properties.4}
\hspace*{-3mm}			\begin{aligned}
				&\vert  (\rho_{\vep}\ast\tau_{2\vep}
      \overline{\bun})(x)-\mathbf{u}(x)\vert\\&=
				\bigg\vert \int_{B_\vep^{d}(x)}{\rho_\vep(x-y)(\overline{\mathbf{u}}(y-2\vep \,\mathrm{e}_d)-\mathbf{u}(x))\,\mathrm{d}y}\bigg\vert 
				\\&=
				\bigg\vert \int_{B_\vep^{d}(x)}{\rho_\vep(x-y)\bigg(\int_0^1{\nabla\overline{\mathbf{u}}(\lambda\,(y-2\vep \,\mathrm{e}_d)+(1-\lambda)\,x): (y-x-2\vep \,\mathrm{e}_d)\,\mathrm{d}\lambda}\bigg)\,\mathrm{d}y}\bigg\vert
				\\&\leq 
				3\,\varepsilon\,\int_0^1{\bigg(\int_{B_{\varepsilon}^{d}(x)}{\rho_{\varepsilon}(x-y)\vert \nabla\overline{\mathbf{u}}(\lambda\,(y-2\vep \,\mathrm{e}_d)+(1-\lambda)\,x)\vert\,\mathrm{d}y}\bigg)\,\mathrm{d}\lambda}
				\\&\leq
				3\,\varepsilon\,\int_0^1{\bigg(\int_{B_{\lambda\varepsilon}^{d}(x-\lambda\,2\vep\,\mathrm{e}_d))}{\rho_{\varepsilon}\Big(\frac{x-y}{\lambda}-2\vep\,\mathrm{e}_d\Big)\vert \nabla\overline{\mathbf{u}}(y)\vert\lambda^{-d}\,\mathrm{d}y}\bigg)\,\mathrm{d}\lambda}
    \\&\leq
				\|\rho\|_{L^\infty(\mathbb{R}^d)}\,3^{d+1}\,\vert B_1^d(0)\vert\,\varepsilon\,
    \fint_{B_{3\lambda\varepsilon}^{d}(  x)}{\vert \nabla\overline{\mathbf{u}}(y)\vert\,\mathrm{d}y} 
				\\&\leq c_d\,
				\|\rho\|_{L^\infty(\mathbb{R}^d)}\,3^{d+1}\,\varepsilon\,M(\vert\nabla\overline{\mathbf{u}}\vert)(x)\,.
			\end{aligned}\hspace*{-15mm}
		\end{align}
		
  \textit{ad (iv).} Follows from (iii) and the $L^r$-$L^r$-stability of the
  Hardy--Littlewood~maximal operator
  $M\colon L^r(\mathbbm{R}^{d})\to L^r(\mathbbm{R}^{d})$.
  
  \textit{ad (v).} Let
  $(\overline{\bun}_n)_{n\in \mathbb{N}}\hspace{-0.1em}\subseteq \hspace{-0.1em}(C^\infty_0(\mathbbm{R}^{d}))^{\ell}$ be such that
  $\overline{\bun}_n\hspace{-0.1em}\to\hspace{-0.1em} \overline{\bun}$ in $(W^{1,r}_0(\mathbbm{R}^{d}))^{\ell}$ $(n\hspace{-0.1em}\to\hspace{-0.1em} \infty)$, \textit{i.e.},
  for every $\delta>0$, there exists $n_0\in \mathbb{N}$ such that
  $\|\overline{\bun}_n-\overline{\bun}\|_{W^{1,r}(\mathbbm{R}^{d})}\leq \delta$ for all
  $n\in \mathbb{N}$ with $n\ge n_0$.  Using ~\eqref{eq:basic_properties.3} together with
  the $L^r$-$L^r$-stability of the Hardy--Littlewood maximal operator
  $M\colon L^r(\mathbbm{R}^{d})\to L^r(\mathbbm{R}^{d})$,
  and~\eqref{eq:basic_properties.1}, for every $n\in \mathbb{N}$ with $n\ge n_0$, $\vep>0$, and a.e.\ $x\in \rplus$,  
  we conclude that
  \begin{align*}
		\begin{aligned}
                  \|\nabla S_{\vep}(\bun) -\nabla\bun \|_{L^r(\rplusd)} &=
                  \|\rho_{\vep}\ast \tau_{2\vep
                  }\nabla\overline{\bun}-\nabla\bun\|_{L^r(\rplusd)}
                  \\
                  &\leq \|\rho_{\vep}\ast
                  \tau_{2\vep}(\nabla\overline{\bun}_n-\nabla\overline{\bun})\|_{L^r(\rplusd)} 
                  \\
                  &\quad+\|\rho_{\vep}\ast \tau_{2\vep
                  }\nabla\overline{\bun}_n-\nabla\overline{\bun}_n\|_{L^r(\rplusd)}
                  \\
                  &\quad+\|\nabla\overline{\bun}_n-\nabla\bun\|_{L^r(\rplusd)}
                  \\
                  &\leq (1+c_d)\,\delta
                  +\|\rho_{\vep}\ast \tau_{2\vep
                  }\nabla\overline{\bun}_n-\nabla\overline{\bun}_n\|_{L^r(\rplusd)}\,.
		\end{aligned}
	\end{align*}
	For every $n\in \mathbb{N}$ and $x\in \rplusd $,  it holds that
        \begin{align*}
		\begin{aligned}
                  &\vert (\rho_{\vep}\ast \tau_{2\vep
                  }\nabla\overline{\bun}_n)(x)-\nabla\overline{\bun}_n(x)\vert
                  \\
                  &\leq \bigg\vert \int_{B_{\vep}^{d}(0)}{\rho_{\vep}
                    (y)(\nabla\overline{\bun}_n(x-2\vep\,\mathrm{e}_d
                    -y)-\nabla\overline{\bun}_n(x))\,\mathrm{d}y}\bigg\vert
                  \\
                  &\leq 3\,\vep\,[\nabla\overline{\bun}_n]_{{\dot C}^{0,1}(\rplusd
                    )} \,,
		\end{aligned}
	\end{align*}
        which implies that
        \begin{align*}
          \limsup_{\vep \to 0+}{	\|\nabla S_{\vep}(\bun) -\nabla\bun \|_{L^r(\rplusd)}}\leq (1+c_d)\,\delta\,. 
	\end{align*}
	By passing to the limit as $\delta\to 0+$, we conclude that
        $\nabla S_{\vep}(\bun) \to \nabla\bun$ in $(L^r(\rplusd))^{\ell\times d}$ $(\vep \to 0^{+})$. 
        The same argument yields that $S_{\vep}(\bun) \to \bun$ in $(L^r(\rplusd))^{\ell}$ $(\vep \to 0^{+})$.  
      \end{proof}
\begin{proof}[Proof (of Proposition~\ref{prop:admissibility}).]
  The proof of this result is a standard application of the same tools employed in the
  previous one, which we sketch below.

  \textit{ad (i).} The smoothness $\rho_{\vep}\ast S_{\vep}(\bun) \in (C^\infty(\rplusd))^{\ell}$
  for all $\vep>0$ is evident. For every $\vep >0$, using
  Proposition~\ref{prop:basic_properties}(i), we find that
  \begin{align*}
    \begin{aligned}
      \mathrm{supp}(\rho_{\vep}\ast S_{\vep}(\bun) )&\subseteq
                  \mathrm{supp}(\rho_{\vep})+\mathrm{supp}(S_{\vep}(\bun) )
                  \\
                  &\subseteq\overline{B_{\vep}^{d}(0)}+\varepsilon\,\mathrm{e}_d+ \rplusd
                  \subseteq\rplusd \,.
		\end{aligned}
  \end{align*}

	\textit{ad (ii).} Follows from Proposition~\ref{prop:basic_properties}(ii) since
        the convolution operator commutes with differential operators.
\end{proof}

\section*{Acknowledgments}

%

LCB acknowledges support by INdAM-GNAMPA, by MIUR, within the project
PRIN20204NT8W4$_{-}$004 Nonlinear evolution PDEs, fluid dynamics and transport equations:
theoretical foundations and applications, and by MIUR Excellence, Department of
Mathematics, University of Pisa, CUP I57G22000700001.

LCB also expresses gratitude to King Abdullah University of Science and Technology (KAUST)
for support and hospitality during part of the preparation~of~the
paper.

\section*{Competing Interests}
On behalf of all authors, the corresponding author states that there is no conflict of interest.

%

\def\ocirc#1{\ifmmode\setbox0=\hbox{$#1$}\dimen0=\ht0 \advance\dimen0
  by1pt\rlap{\hbox to\wd0{\hss\raise\dimen0
  \hbox{\hskip.2em$\scriptscriptstyle\circ$}\hss}}#1\else {\accent"17 #1}\fi}
  \def\polhk#1{\setbox0=\hbox{#1}{\ooalign{\hidewidth
  \lower1.5ex\hbox{`}\hidewidth\crcr\unhbox0}}} \def\cprime{$'$}

%

\end{document}